\documentclass[11pt,reqno]{amsart}
\usepackage[margin=1in]{geometry}
\usepackage{amsmath, amsthm, amssymb,mathtools}
\usepackage[colorlinks=true, pdfstartview=FitV, linkcolor=blue,citecolor=blue, urlcolor=blue]{hyperref}
\usepackage[T1]{fontenc}
\usepackage[utf8]{inputenc}

\usepackage{graphicx}
\usepackage{subcaption}

\usepackage{array}
\usepackage{verbatim}
\usepackage{caption}

\usepackage{multirow}

\usepackage{xcolor}

\usepackage[novbox]{pdfsync}

\usepackage{parskip} \newtheorem{proposition}{Proposition}
\newtheorem{theorem}[proposition]{Theorem}

\newtheorem{lemma}[proposition]{Lemma}
\theoremstyle{definition}

\newtheorem{remark}[proposition]{Remark}

\numberwithin{equation}{section}

 \def\R{\mathbb{R}}
  \def\Z{\mathbb{Z}}
  \def\Q{\mathbb{Q}}
  \def\C{\mathbb{C}}
  \def\N{\mathbb{N}}

\newcommand{\A}{\mathcal{A}}

\newcommand{\B}{\mathcal{B}}

\renewcommand{\H}{{\mathcal{H}}}
\newcommand{\U}{{\mathcal{U}}}
\newcommand{\D}{{\mathcal{D}}}
\newcommand{\K}{{\mathcal{K}}}
\newcommand{\ee}{\mathrm{e}}
\newcommand{\jj}{J}
\newcommand{\p}{\partial}
\renewcommand{\d}{\mathrm{d}}
\renewcommand{\div}{\mathrm{div}}
\newcommand{\re}{\mathrm{Re}}
\newcommand{\im}{\mathrm{Im}}

\newcommand{\eqnb}{\begin{equation}}
\newcommand{\eqne}{\end{equation}}

\title[Nonsymmetric spiral vortex sheets]{Existence of nonsymmetric logarithmic spiral vortex sheet solutions to the 2D Euler equations}

\author{Tomasz Cie\'{s}lak}
\address{Institute of Mathematics, Polish Academy of Sciences, 00-656 Warsaw, Poland\\ \newline and Faculty of Mathematics and Information Sciences, Warsaw University of Technology, 00-662 Warsaw, Poland}
\email{cieslak@impan.pl}

\author{Piotr Kokocki}
\address{Faculty of Mathematics and Computer Science, Nicolaus Copernicus University, 87-100 Toru\'n, Poland}
\email{pkokocki@mat.umk.pl}

\author{Wojciech S. O\.za\'nski}
\address{Institute of Mathematics, Polish Academy of Sciences, 00-656 Warsaw, Poland\\ \newline and Florida State University, Tallahassee, FL 32306, USA}
\email{wozanski@fsu.edu}

\begin{document}

\date{\today}


\begin{abstract}
We consider solutions of the 2D incompressible Euler equation in the form of $M\geq 1$ concentric logarithmic spirals. We prove the existence of a generic family of spirals that are nonsymmetric in the sense that the  angles of the individual spirals are not uniformly distributed over the unit circle. Namely, we show that if $M=2$ or $M\geq 3 $ is an odd integer such that certain non-degeneracy conditions hold, then, for each $n \in \{ 1,2 \}$, there exists a logarithmic spiral with $M$ branches of relative angles arbitrarily close to $\overline\theta_{k} = kn\pi/M$ for $k=0,1,\ldots , M-1$, which include halves of the angles of the Alexander spirals. We show that the non-degeneracy conditions are satisfied if $M\in \{ 2, 3,5,7,9 \}$, and that the conditions hold for all odd $M>9$ given a certain gradient matrix is invertible, which appears to be true by numerical computations.
\end{abstract}

\maketitle

\section{Introduction}\label{sec_intro}

We consider vortex sheets solutions to the 2D Euler equation in the form of $M\geq 1$ concentric logarithmic spirals, described by
\eqnb\label{spirals}
\begin{cases}
Z_m(\theta , t) = t^\mu \ee^{a (\theta - \theta_m )} \ee^{i\theta },\\
\Gamma_m (\theta ,t) = g_m t^{2\mu -1 }  \ee^{2a (\theta - \theta_m )},
\end{cases}
\eqne
where $t>0$, $\theta \in \R$, and
\begin{equation*}
a>0, \quad \mu \in \R , \quad g_m \in \R\setminus\{0\}, \quad \theta_m \in \R
\end{equation*}
for $m=0,\ldots , M-1$, denote parameters of the spiral, with the assumption that 
\eqnb\label{order}
0=\theta_{0}< \theta_1 < \theta_2 <\ldots < \theta_{M-1}.
\eqne

In the above equations $Z_m (\theta , t)$ parametrizes the $m$-th spiral $\Sigma_m(t)\subset \C$, while $\Gamma_m (\theta ,t )$ parametrizes its vorticity distribution. Namely, the vorticity of the associated flow $v(t)$ is assumed to be of the form 
\begin{equation}\label{weak-vort}
\mathrm{curl} \, v (t) = \sum_{k=0}^{M-1}\gamma (t)\delta_{\Sigma_{k}(t)},\quad t>0,
\end{equation}
in the sense of distributions, where $\gamma(t) $ denotes the density function of vorticity along the spirals $\Sigma (t)\coloneqq \bigcup_{k=0}^{M-1} \Sigma_k (t)$ for a given $t>0$, which can be described by $\gamma(Z_{k}(\theta,t),t)\coloneqq {\p_\theta \Gamma_{k}(t,\theta)}\left|{\p_\theta  Z_{k}(\theta,t)}\right|^{-1}$ for $k=0,\ldots , M-1$, $\theta \in \R$.

The authors' previous work \cite{cko} characterizes when the spirals described by \eqref{spirals} are weak solutions to the Euler equations,
\eqnb\label{euler_intro}
\begin{split}
\p_t v + v\cdot\nabla v + \nabla p &=0, \\
\div \, v &=0,
\end{split}
\eqne
on $\R^2 \times (0,\infty )$. Here the notion of the \emph{weak solution of \eqref{euler_intro}} refers to a vector field $v\in L^2_{loc} (\R^2 \times (0,\infty ))$ that is weakly divergence-free, that is $\int_{\R^{2}}v\cdot\nabla\psi = 0$ for all $\psi\in C_{0}^{\infty}(\R^{2})$, and the weak formulation of \eqref{euler_intro},
\begin{equation*}
\int_0^\infty \int_{\R^2} \left( v \cdot \p_t \varphi + \sum_{1\le i,j\le 2}v_i v_j \p_i \varphi_j \right) =0,
\end{equation*}
holds for all divergence-free $\varphi \in C_0^\infty (\R^2 \times (0, \infty );\R^{2})$. The result of \cite{cko} shows that if 
\begin{equation*}
v (z,t) =  t^{\mu -1} w \left( \frac{z}{t^\mu } \right),
\end{equation*}
where the profile function $w$ is given in the polar coordinates by 
\begin{equation*}
w(z) \coloneqq  \ee^{i\theta } \sum_{k=0}^{M-1} \frac{2ag_{k} }{r(a-i)} \left(  r^{\frac{2a}{a+i} }\ee^{A  (\theta_k -\theta ) } \frac{\ee^{2\pi \jj (r,\theta ,k)A }}{1-\ee^{2\pi A}} \right)^*,
\end{equation*}
where $z=r\ee^{i\theta }\in \C \setminus (\Sigma \cup \{ 0 \})$ and $
\jj =\jj (r,\theta ,k ) \coloneqq \min \left\lbrace j\in \Z \colon a(2\pi j + \theta_k - \theta ) + \ln r >0 \right\rbrace$ denotes the \emph{winding number of the spiral}, then the equation \eqref{weak-vort} holds in the weak sense \cite[Theorem~1.8(iii)]{cko}. Moreover, such $v$ is a weak solution of \eqref{euler_intro} if and only if 
\begin{equation}\label{eq-disc}
\sum_{k=0}^{M-1} \mathcal{A}_{mk}(a,\Theta)g_{k} = -\frac{ \sinh (\pi A)}{2a^2 }   \left( 1+a^2 -2\mu (1-ai ) \right)
\end{equation}
for $m=0,\ldots , M-1$ \cite[Theorem~1.3]{cko}, where 
\eqnb\label{theta_2pi}
\Theta \coloneqq (\theta_1 , \ldots , \theta_{M-1} )\in (0,2\pi )^{M-1},
\eqne 
\eqnb\label{def_of_A}
A\coloneqq - \frac{2ai}{a+i},
\eqne
and 
\begin{equation}\label{matrix-a}
\mathcal{A}_{mk}(a,\Theta):= \ee^{A(\theta_{k}-\theta_{m})}
\left\{\begin{aligned}
&\ee^{-\pi A}, && \text{if} \ \ \theta_{k}>\theta_{m},\\
&\cosh(\pi A), && \text{if} \ \ \theta_{k}=\theta_{m},\\
&\ee^{\pi A}, && \text{if} \ \ \theta_{k}<\theta_{m}
\end{aligned}\right. 
\end{equation}
for $k,m \in \{ 0, \ldots ,  M-1 \}$.  Here we have identified a complex function $v(z,t) = v_1 (x,y,t) + i v_2( x,y,t)$ with the 2D vector field $(v_1(x,y,t),v_2(x,y,t))$, where $z=x+iy$.

We note that, although the structure of the discrete system \eqref{eq-disc} is rather complicated, it does enjoy some symmetries. For example, it is invariant under the angle shift $\theta_{k}\mapsto \theta_{k}+\alpha$ and permutations $\sigma$ of the parameters $\theta_{k}\mapsto \theta_{\sigma(k)}$, $g_{k}\mapsto g_{\sigma(k)}$ for $ 0\le k \le M-1$, which we have already used to assume that $\theta_0 =0$ and that the angles are ordered  \eqref{order}. Hence, if we define the set 
\eqnb\label{def_of_U} \U \coloneqq \{\Theta=(\theta_1 , \theta_2 , \ldots , \theta_{M-1})\in(0,2\pi )^{M-1} \ | \ 0<\theta_{1}<\theta_{2}
<\ldots<\theta_{M-1} <2\pi \},
\eqne
then the problem of existence of logarithmic spiral solutions to the Euler equations \eqref{euler_intro} reduces to the problem of choice of parameters $a>0$, $g_0,\ldots , g_{M-1}  \in \R\setminus \{ 0 \}$, $\Theta \in \U$ satisfying \eqref{eq-disc} for all $0\le m\le M-1$.

So far only one generic example of logarithmic spirals has been known. It corresponds to the case of equally distributed angles $\theta_0,\ldots , \theta_{M-1}$, with the same weights $g_k$, namely
\eqnb\label{alex}
\theta_k = \frac{2k\pi }{M}, \qquad g_k = g \in \R\setminus \{ 0 \},
\eqne
which are commonly referred to as  the \emph{Alexander spirals} \cite{alexander}. The case $M=1$ is often referred to as the \emph{Prandtl spiral}, and its history goes back to the 1922 work of Prandtl \cite{prandtl22} (see also \cite{prandtl05,prandtl20}), who introduced them as a model of  wingtip  vortices. In this particular case \eqref{eq-disc} becomes 
\eqnb\label{prandtl_coeffs}
-2a^2 g \coth (\pi A)=a^2+1-2\mu + 2a\mu i,
\eqne
see \cite[Corollary~1.4]{cko}.

This gives 2 equations for 3 unknowns: $a$, $g$, $\mu$, and so one expects infinitely many solutions. In fact, one can see that for all sufficiently large $a$ there exists unique choice of $\mu, g $ such that \eqref{prandtl_coeffs} holds, see the discussion following \cite[Corollary~1.4]{cko} for details.

We also note a recent result of the authors \cite{cko_instab} showing linear instability of
Alexander spirals.

Besides Alexander spirals \eqref{alex} and the well-known stationary discontinuous shear flows, there are  few known vortex sheet solutions (i.e. with vorticity supported on a curve) satisfying rigorously the 2D Euler equations. One of them is the Prandtl-Munk vortex sheet introduced in \cite{m}, which was shown to satisfy the nonhomogeneous Euler equations in \cite{nls}. 

Actually, as shown recently in \cite{gpsi}, under mild regularity assumptions, the only stationary (or uniformly rotating) vortex sheets supported on a finite union of disjoint curves solving 2D Euler is trivial, i.e. it is a union of concentric circles. In \cite{ps} an explicit construction of intersecting Munk-like vortex sheets is given (see also \cite{pls}). The first  rigorous example of time-dependent self-similar vortex sheets was provided in the authors' previous work \cite{cko}, see also \cite{js}. Moreover, up to our knowledge, the only other rigorous examples of nontrivial time-dependent global-in-time  vortex sheet solutions are the perturbation of steady solutions, see \cite{co} or \cite{dr}. 

The purpose of this work is to provide a new generic family of  vortex sheets in the form of logarithmic spirals that are nonsymmetric in the sense that the angles $\theta_k$ are not uniformly distributed, that is, $\theta_{k+1}- \theta_k\ne \theta_k - \theta_{k-1}$ for at least one $k\in \{ 1, \ldots , M-1 \}$ (where we identify $\theta_M\coloneqq 2\pi$; recall also \eqref{order}). In fact, the question of the existence of nonsymmetric spirals has remained open, and \cite[Section~5]{EL} provided some evidence suggesting existence of such spirals. 

Before stating our main result, we emphasize that the existence of nonsymmetric spirals reduces to finding the angles $\Theta = (\theta_1 , \ldots , \theta_{M-1} ) \in \R^{M-1}$ such that \eqref{eq-disc} holds for some $a>0$, $\mu \in \R$, and $g_0, \ldots , g_{M-1} \in \R\setminus \{ 0\}$. 

\begin{theorem}[The main result]\label{thm_main}
Given $M\in \{ 2 , 3,  5 , 7 , 9 \} $ and $n\in \{1,2 \}$, let \eqnb\label{angles_possible}
\overline{\Theta }= \left( n \frac{\pi }{M} , 2n \frac{\pi }{M}, \ldots , (M-1) n \frac{\pi }{M}    \right)
 \eqne
and $\varepsilon >0$. Then there exists $a_0 > 0$ such that for every $a\geq a_0$ there exists a unique choice of 
\begin{equation}\label{sol_vars}
\Theta = (\theta_{1},\ldots,\theta_{M-1}) \in \R^{M-1}, \quad (g_{0},g_{1},\ldots,g_{M-1})\in \R^{M}, \quad \mu \in \R
\end{equation}
such that \eqref{eq-disc} holds and
\eqnb\label{halves} \left| \Theta  - \overline{\Theta }  \right|  \leq \varepsilon .
\eqne
In particular (as discussed above), then there exist weak solutions of the 2D incompressible Euler equations in the form of logarithmic spirals \eqref{spirals} with angles \eqref{halves}.
\end{theorem}

Since there are few known examples of time-dependent vortex sheet solutions to the 2D Euler equations \eqref{euler_intro}, the spirals provided by Theorem~\ref{thm_main} are of interest. Surprisingly, besides the use of robust analytical tools, the steps leading to Theorem~\ref{thm_main} (i.e. Steps 1--5 below) involve advanced matrix analysis (see Section~\ref{sec_dets} and Appendix~\ref{sec_app1}), combinatorial tools (see Proposition~\ref{prop_aa}), asymptotic expansions (see \eqref{expan1}--\eqref{expan2} and \eqref{asym-11aa}), as well as an application of the field of algebraic numbers and transcendentality of $\pi$ (see Proposition~\ref{prop-e_2_1_n_0}). 

 \begin{remark}
 We note that for $n\geq 3$ the claim of Theorem~\ref{thm_main} regarding solutions of the system \eqref{eq-disc} still holds. However, its interpretation as producing weak solutions of the Euler equations is no longer clear for such $n$, as the assumption $\Theta  \in (0, 2\pi )^{M-1}$ fails (recall \eqref{theta_2pi}). 
 \end{remark}

We also note that Theorem~\ref{thm_main} applies also to any odd integer $M\ge 3$ provided a certain non-degeneracy condition holds, which we discuss below.
\begin{theorem}\label{cor_suff_cond}
Given an odd integer $M \geq 3$ and $n\in \{1,2 \} $ let $\mathcal{C}$ be the matrix given by 
\eqnb\label{matrix-c}
\mathcal{C}_{lm}=\left\{\begin{aligned}
& 2\sin^{2}(\overline\theta_{1}) - (-1)^{m}\sin(2\overline\theta_{1})\sin(2 \overline\theta_{m}) && \text{ if } \quad l=m,\\[5pt]
& (-1)^{m+1}\sin(2 \overline\theta_{1})(\sin(2 \overline\theta_{m}) + (-1)^{l}\sin(2 \overline\theta_{l-m})) && \text{ if } \quad l<m,\\[5pt]
& (-1)^{m+1}\sin(2 \overline\theta_{1})(\sin(2 \overline\theta_{m}) - (-1)^{l}\sin(2 \overline\theta_{l-m})) && \text{ if } \quad m<l
\end{aligned}\right.
\eqne
for $1\le l,m \le M-1$, where $\overline\Theta = (\overline\theta_{1}, \overline\theta_{2},\ldots, \overline\theta_{M-1})$ is given by \eqref{angles_possible}. Then the assertions of Theorem~\ref{thm_main} remain valid provided the matrix $\mathcal{C}$ is invertible.
\end{theorem}

The invertibility of $\mathcal{C}$ mentioned in the above theorem is in fact a nondegeneracy assumption of an argument based on the Implicit Function Theorem. The argument (which is sketched in Section~\ref{sec_sketch} below) gives the claim of Theorem~\ref{thm_main} as $\mathcal{C}$ is invertible for $M\in \{ 2,3,5,7,9 \}$, which we verify in Sections~\ref{sec_case2}--\ref{sec_case_mleq9} below. Numerical computations suggest that it is invertible  also for all odd $M>9$, although we do not provide a rigorous verification.

\begin{remark}
We consider angles $\overline{\Theta }= (\overline{\theta}_1 , \ldots , \overline{\theta}_{M-1})$ of the form \eqref{angles_possible} because of two properties which we use to show their non-degeneracy (see Step 3 and \eqref{naF_is_inver} below). First is that $\theta_k=k\theta_1 $ for each $k\in \{ 1, \ldots , M-1\}$, which we use in various computations in Section~\ref{sec_nondeg_theta} and Section~\ref{sec_case_Mgeq3}, and the second is that $M\theta_1=n\pi$, which we use to obtain \eqref{rr2}--\eqref{rr1}. We note that we do not claim that \eqref{angles_possible} is the only possible choice for which Theorem~\ref{thm_main} is valid.
\end{remark}

We note that, by taking $n=2$ the uniqueness claim of the theorem gives Alexander spirals \eqref{alex}, in which case $g_0=g_1=\ldots = g_{M-1}=g$ for some $g\ne 0$ and $\Theta = \overline{\Theta}$. Moreover, taking $n=1$ gives angles $\Theta \coloneqq \Theta(a)$ that are approximately equal to halves of the angles of Alexander spirals as $a\to \infty$, see Table~\ref{table_angles} below.

\begin{center}
\begin{tabular}{ |>{\centering}m{1cm}|>{\centering}m{3cm}|>{\centering}m{3cm}|>{\centering}m{3cm}|m{3cm}|} 
 \hline
  $M$ &   $2$ & $3$ & $5$ & \hspace{1.25cm} $7$  \\
  \hline
 Sketch & $\begin{array}{c}  \includegraphics[width=2cm]{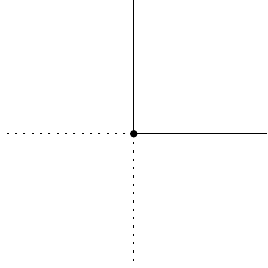}\\ n=1 \end{array}$ & $\begin{array}{c}  \includegraphics[width=2cm]{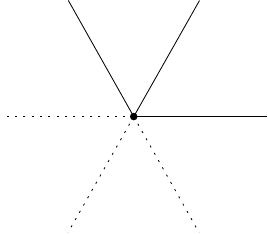}\\ n=1 \\
\includegraphics[width=2cm]{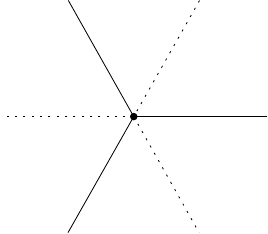}\\ n=2   \end{array}$ & $\begin{array}{c}  \includegraphics[width=2cm]{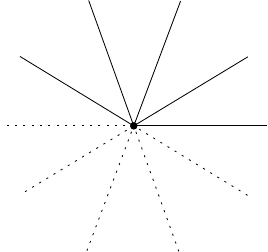}\\ n=1 \\
\includegraphics[width=2cm]{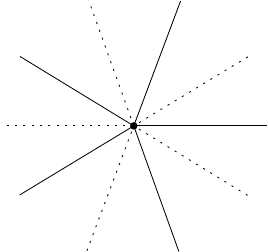}\\ n=2   \end{array}$ & \hspace{0.15cm} $\begin{array}{c}  \includegraphics[width=2cm]{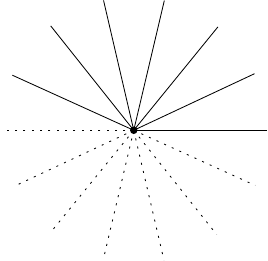}\\ n=1 \\
\includegraphics[width=2cm]{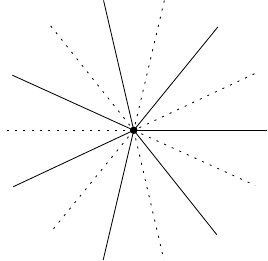}\\ n=2   \end{array} $ \\
  \hline
\end{tabular}
 \nopagebreak \vspace{5pt}
 \captionof{table}{A sketch of the angles of the spirals provided by Theorem~\ref{thm_main} for $M\leq 7$. Note that taking $n=2$ gives the Alexander spirals.}\label{table_angles} 
\end{center}

\begin{remark}\label{rem_nontrivial}
In the case $n=1$, the choice of the angles $\Theta=\Theta (a)$ and the weights $g_0=g_{0}(a), \ldots , g_{M-1}=g_{M-1}(a)$ in Theorem~\ref{thm_main} is nontrivial in the sense that $(\Theta , g_0,g_1,\ldots , g_{M-1}) \ne (\overline{\Theta} , g,\ldots , g ) $ for any $g\in \R\setminus \{ 0 \}$. In fact, for such $(\Theta , g_0,g_1,\ldots , g_{M-1})$ the left-hand side of \eqref{eq-disc} becomes a finite geometric series, which can be calculated explicitly, and gives different values for different $m\in \{ 0, \ldots , M-1 \}$ (see Appendix~\ref{sec_app2} for the details), which implies that \eqref{eq-disc} cannot hold.  

On the other hand, we will show that $g_1/g_0,\ldots , g_{M-1}/g_0$ do converge to $1$ as $a\to \infty$, see \eqref{expan2} below. Moreover, \eqref{def_of_-1} below provides asymptotic expansion of this convergence. In particular, since the value of  $g_0$ is chosen separately (see Step~5 below), the vector $(g_0,g_1, \ldots , g_{M-1} )$ has all entries approximately equal to $g_0$.
\end{remark}
\begin{remark}\label{rem_m_geq4}
We note that no nonsymmetric spirals can be obtained from Theorem~\ref{thm_main} for even $M>2$. Interestingly, this is a consequence of two separate arguments: the breakdown  of an admissibility criterion of $\overline{\Theta } =\overline{\theta}_1 (1, \ldots , M-1 ) $ (see \eqref{admissible} below), and the geometric constraint that $\overline{\theta}_1 = k\pi/2 $ for some $k\in \Z$, which in light of the ordering \eqref{order} implies that the resulting spiral must be symmetric. We discuss this issue in detail in Section~\ref{sec_m_geq4}. 
\end{remark}

In order to expose the main difficulty of Theorem~\ref{thm_main}, we note that \eqref{eq-disc} is equivalent to the vector equation 
\eqnb\label{matrix_eq_main}
\A g =  -\frac{ \sinh (\pi A)}{2a^2}\left( 1+a^2 -2\mu (1-ai ) \right) \begin{pmatrix}
1 \\ \vdots \\ 1 
\end{pmatrix},
\eqne
where we used \eqref{matrix-a}, which can be written in the matrix notation as 
\[
\A =  \begin{pmatrix}  
\frac{r+r^{-1}}2  & r \frac{r_1 }{r_2} & r \frac{r_1}{r_3} & \ldots & r \frac{r_1}{r_M} \\ r^{-1} \frac{r_2}{r_1} & \frac{r + r^{-1}}2 & r \frac{r_2}{r_3} &\ldots & r \frac{r_2}{r_M}  \\
r^{-1} \frac{r_3}{r_1} & r^{-1} \frac{r_3}{r_2} & \frac{r+r^{-1} }2 & \ldots & r \frac{r_3}{r_M} \\
\vdots & \vdots & \vdots & \ddots & \vdots \\
r^{-1} \frac{r_M}{r_1} & r^{-1} \frac{r_M}{r_2} &r^{-1} \frac{r_M}{r_3} &\ldots & \frac{r+r^{-1}}2
\end{pmatrix},
\]
with
\eqnb\label{def_r_rk}
r\coloneqq \ee^{-A\pi },\qquad r_k\coloneqq \ee^{-A\theta_{k-1}}, \quad k=1,\ldots , M.
\eqne
It can be verified that $\A$ is nonsingular for every $\Theta\in\mathcal{U}$ and $a>0$, see Appendix~\ref{sec_app1} for details. Therefore, at first sight, finding solutions $g\in (\R \setminus \{ 0 \} )^M$, $a>0$, $\mu \in \R$,  of \eqref{matrix_eq_main} might seem immediate: given any $a,\Theta, \mu$ let $g=(g_0, \ldots , g_{M-1})$ be the solution of the linear system \eqref{matrix_eq_main}. The problem is that such approach only guarantees that $g\in \C^{M}$, while we must have that $g\in (\R\setminus\{0 \})^{M}$. \\

This is in fact the main difficulty of finding solutions to \eqref{matrix_eq_main}.\\

One could expect that a more careful choice of $a$, $\Theta$ and $\mu$ could result in $g\in (\R\setminus\{0 \})^{M}$, and the main purpose of this paper is to provide a family of nontrivial examples of such choices. Namely, we present a construction which results in $\Theta$ approximately equal to $\overline{\Theta }$ (given by \eqref{angles_possible}, for any fixed $n\in \{ 1,2 \} $).

To this end we observe that dividing \eqref{matrix_eq_main} by $g_0$ and subtracting the first equation of \eqref{matrix_eq_main} from all the remaining $M-1$ equations gives
\eqnb\label{system_for_g'}
\B g' = -b,
\eqne
where $g'\coloneqq (g_1, \ldots , g_M )/g_0$, $b\in \C^{M-1}$ is the vector given by 
\begin{align*}
b:=\left( \frac{r^{-1}r_2-cr_1}{r_1},  \frac{r^{-1}r_3-cr_1}{r_1},  \ldots,  \frac{r^{-1}r_M-cr_1}{r_1} \right) ,
\end{align*}
and $\B$ is the bottom-right $(M-1)\times (M-1)$ submatrix of 
\eqnb\label{oA}
\mathcal{A}_{0} \coloneqq 
\begin{pmatrix}
1 & 0  & \ldots & 0\\
-1 & 1 & \ldots & 0\\
-1& 0 & \ldots & 0\\
\vdots & \vdots & \ddots &\vdots \\
-1 &0 & \ldots & 0
\end{pmatrix} \A = \begin{pmatrix}  
\frac{r+r^{-1}}2  & r \frac{r_1 }{r_2} & r \frac{r_1}{r_3}  & \ldots & r \frac{r_1}{r_M} \\
\frac{r^{-1} r_2-c r_1}{r_1}&  \frac{cr_2-rr_1}{r_2} & r\frac{r_2-r_1}{r_3}  &\ldots &  r\frac{r_2-r_1}{r_M}  \\
\frac{r^{-1} r_3-c r_1}{r_1}& \frac{r^{-1}r_3-rr_1}{r_2} & \frac{cr_3 - rr_1 }{r_3}& \ldots & r \frac{r_3-r_1}{r_M}\\
\vdots & \vdots &  \vdots &\ddots & \vdots \\
 \frac{r^{-1} r_M-c r_1}{r_1}& \frac{r^{-1}r_M-rr_1}{r_2} & \frac{r^{-1}r_M-rr_1}{r_3} &\ldots & \frac{cr_M-rr_1}{r_M}
\end{pmatrix},
\eqne
where we used the notation  
\[
c\coloneqq \frac{r^{-1}+r}{2} = \cosh (\pi A).
\]

The first main difficulty of finding solutions \eqref{sol_vars} to \eqref{matrix_eq_main} is to find $a$ and $\Theta (a)$ such that \eqref{system_for_g'} has a unique solution $g'\in (\R \setminus \{ 0 \} )^{M-1}$. Suppose for the moment that we have succeeded in finding such $a$ and $\Theta(a)$. Then \eqref{matrix_eq_main} will be satisfied if $g_0\in \R\setminus \{ 0\}$ and $\mu \in \R$ are chosen so that the first equation of \eqref{matrix_eq_main} holds, that is 
\begin{equation}\label{eq-g-n}
g_0\underbrace{\left( \frac{ac}{s} + \frac{ar}{s} \left( \frac{g'_1}{r_2} + \ldots + \frac{g'_{M-1}}{r_M} \right) \right) }_{=:D} = \mu\left(i - \frac{1}{a}\right) + \frac{1+a^{2}}{2a},
\end{equation}
where we have set $s\coloneqq (r-r^{-1})/2=-\sinh (\pi A)$. This is the second main difficulty of the problem. In fact, such $g_0,\mu$ exist if and only if $D$ and $i - \frac{1}{a}$ are linearly independent (as vectors in the complex plane). Then representing $(1+a^2)/2a^2$ as a linear combination of $D$ and $i - \frac{1}{a}$ gives the unique choice of $g_0,\mu \in \R$. In such case note that we must have $g_0\ne 0$ (as required), since $\frac{1+a^2}{2a}$ and  $i - \frac{1}{a}$ are linearly independent.\\

\subsection{Proof of Theorem~\ref{thm_main}}\label{sec_sketch}

In this section we describe the above idea in details, proving Theorem~\ref{thm_main}, except for claims which are verified in detail in Sections~\ref{sec_dets}--\ref{sec_choice_g_mu} below.

\noindent\texttt{Step 1.} We use the Cramer method to deduce that the solution $g'$ of \eqref{system_for_g'} is
\begin{equation}\label{def_of_H}
H (a,\Theta ) :=\frac{1}{\K (a,\Theta ) } \left( \H_1 (a,\Theta ) , \H_2 (a,\Theta ) , \ldots , \H_{M-1} (a,\Theta ) \right) 
\end{equation}
for all sufficiently large $a >0$, and all $\Theta =(\theta_1, \ldots , \theta_{M-1} )\in \R^{M-1}$ sufficiently close to $\overline{\Theta}$ (recall \eqref{angles_possible}), where
\eqnb\label{def_of_K,Hl}
\begin{split}
\K (a,\Theta) &\coloneqq \frac{\ee^{\pi A} + (-1)^{M}\ee^{-\pi A }}{2} + (-1)^{M}\sum_{k=1}^{M-1}(-1)^{k}\ee^{A(\theta_{k}-\pi)}, \\
\H_l (a,\Theta ) &\coloneqq \frac{ \ee^{\pi A} - (-1)^{M}\ee^{-\pi A} }{2} 
+\sum_{k=0}^{l-1}(-1)^{l+k}\ee^{A(\theta_{k}-\theta_{l}+\pi)} +(-1)^{M}\sum_{k=l}^{M-1}(-1)^{l+k}\ee^{A(\theta_{k}-\theta_{l}-\pi)}.\\
\end{split}
\eqne

To this end we let $\B_l$ ($l=1,\ldots , M-1$) denote $\B$ with the $l$-th column replaced by $-b$, that is
\eqnb\label{def_Bl}
\B_l \coloneqq \left(  \begin{array}{ccccc|c|ccc}
 \frac{cr_2-rr_1}{r_2} & \frac{rr_2-rr_1}{r_3}  & \frac{rr_2-rr_1}{r_4}  &\ldots &\frac{rr_2-rr_1}{r_l}& -\frac{r^{-1}r_2-cr_1}{r_1}&\frac{rr_2-rr_1}{r_{l+2}} & \ldots & \frac{rr_2-rr_1}{r_M}  \\
 \frac{r^{-1}r_3-rr_1}{r_2} & \frac{cr_3 - rr_1 }{r_3} & \frac{rr_3-rr_1}{r_4} & \ldots &\frac{rr_3-rr_1}{r_l}& -\frac{r^{-1}r_3-cr_1}{r_1}&\frac{rr_3-rr_1}{r_{l+2}} & \ldots &  \frac{rr_3-rr_1}{r_M} \\
 \frac{r^{-1}r_4-rr_1}{r_2} & \frac{r^{-1} r_4 - rr_1 }{r_3} & \frac{cr_4-rr_1}{r_4} & \ldots &\frac{rr_4-rr_1}{r_l}& -\frac{r^{-1}r_4-cr_1}{r_1}&\frac{rr_4-rr_1}{r_{l+2}}& \ldots &  \frac{rr_4-rr_1}{r_M} \\
 \vdots & \vdots & \vdots &\ddots & \vdots &\vdots &\vdots & \ddots & \vdots \\
 \frac{r^{-1}r_M-rr_1}{r_2} & \frac{r^{-1}r_M-rr_1}{r_3}& \frac{r^{-1}r_M-rr_1}{r_4} &\ldots &\frac{rr_M-rr_1}{r_l}& -\frac{r^{-1}r_M-cr_1}{r_1}&\frac{r^{-1}r_M-rr_1}{r_{l+2}}& \ldots & \frac{cr_M-rr_1}{r_M}
 \end{array} \right),
\eqne
where the column inside the vertical lines is the $l$-th column. We emphasize that each of $\A$, $\mathcal{A}_{0}$, $\B$, $\B_l$, $r$, $r_k$, $c$ depends on $(a,\Theta )$.  We prove in Propositions~\ref{prop_B}~and~\ref{prop_Bl} that
\eqnb\label{toshow_dets}
\begin{split}  
\det \B &= \K \sinh^{M-2}(\pi A), \\
\det \B_l &= \H_l \sinh^{M-2}(\pi A) .
\end{split}
\eqne
We let 
\begin{align}\label{eq-l-1b}
&\overline{\K}(\Theta ) \coloneqq \lim_{a\to \infty } \K (a, \Theta ) = \frac{1+(-1)^M}{2}+(-1)^M\sum_{k=1}^{M-1}(-1)^{k}\ee^{-2i\theta_k}, \\ \label{eq-l-2b}
&\overline{\H}_l(\Theta ) \coloneqq \lim_{a\to \infty } \H_l (a, \Theta ) = \frac{1-(-1)^M }{2}
+ \sum_{k=0}^{l-1}(-1)^{l+k}\ee^{2i(\theta_l-\theta_k)}+(-1)^M\sum_{k=l}^{M-1}(-1)^{l+k}\ee^{2i(\theta_l-\theta_{k})},
\end{align}
and we observe that the above limits, as $a\to \infty $, are uniform with respect to $\Theta \in \U$ on each compact subset of $\U$ (recall \eqref{def_of_U}).  We define the set of admissible angles as
\[
\D\coloneqq \{\Theta\in \mathcal{U} \colon   \overline{\K} (\Theta)\neq 0\},
\]
and we let
\begin{equation*}
\overline{H}(\Theta)\coloneqq  \frac{1}{\overline{\K} (\Theta)}\left( \overline{\H}_1 (\Theta), \overline{\H}_2 (\Theta), \ldots, \overline{\H}_{M-1} (\Theta) \right),\quad \Theta\in\mathcal{D}.
\end{equation*}
Thus, since $\D\subset \R^{M-1}$ is open, we see that
\eqnb\label{convergence}
H(a,\Theta ) \to \overline{H} (\Theta )\qquad \text{ as } a\to \infty,
\eqne
uniformly with respect to $\Theta$ on compact subsets of $\D$ (recall that $A= -2ai/(a+i) \to -2i$ as $a\to \infty$).
We will use the notation
\[\begin{split}
F(a,\Theta ) &\coloneqq \im\, H(a,\Theta ),\hspace{3cm} G(a,\Theta ) \coloneqq \re\, H(a,\Theta ),  \\
\overline{F} (\Theta ) &\coloneqq \im\, \overline{H} (\Theta ), \hspace{3.76cm} \overline{G} (\Theta ) \coloneqq \re\, \overline{H} (\Theta ).
\end{split}
\]
Note that, for brevity, we use the convention that a bar above a quantity denotes its limiting value as $a\to \infty$, e.g. $\overline{\Theta } = \lim_{a\to \infty } \Theta (a)$.

Hence, since $\overline{\Theta } = \frac{n\pi }{M} \left( 1, 2,\ldots, M-1 \right)$ (recall \eqref{angles_possible}) is admissible, that is 
\eqnb\label{admissible}
\overline{\Theta } \in \D, 
\eqne
which we prove in Section~\ref{sec_nondeg_theta} below, we obtain that ${\K} (a,\Theta )\ne 0$ for sufficiently large $a>0$ and $\Theta \in \D$ sufficiently close to $\overline{\Theta }$. Therefore, for such $a,\Theta$, we see that $\B $ is invertible and $H(a,\Theta )$ is the unique solution of \eqref{system_for_g'}, due to the Cramer method.\\

\noindent\texttt{Step 2.} We show in Section~\ref{sec_conv_claims} that the limit \eqref{convergence} is stronger in the sense that for every closed ball $D(\overline\Theta,\varepsilon)\subset \D$ there exists $\delta>0$ such that the function $\widetilde{F} \colon [0,\delta ] \times B(\overline\Theta, \varepsilon ) \to \R^{M-1}$,
\eqnb\label{def_of_tildeF}
\widetilde{F}(t, \Theta ) \coloneqq \begin{cases}
F(1/t, \Theta ) \qquad &\text{ for }t\in (0,\delta ],\\
\overline{F}(\Theta ) &\text{ for } t=0,
\end{cases}
\eqne
is a $C^1$ map. To this end we prove (in Section~\ref{sec_conv_claims} below) that the limit 
\eqnb\label{R+iI}
R(\Theta )+iI (\Theta )\coloneqq \lim_{a\to \infty } \left[ - a^2 \p_a H (a,\Theta ) \right]
\eqne
exists and is uniform with respect to $\Theta $ on every compact subset of $\D$. We also show that
\eqnb\label{limit_at_angles}
(R+iI) (\overline{\Theta }) = 
-i\pi \frac{((-1)^{l} \ee^{2il\theta_1} -1 )(1+\ee^{2i\theta_1 })}{\sin (2 \theta_1 ) }  
\eqne
for odd $M\geq 3$ and each $l\in\{1,\ldots, M-1\}$.

\noindent\texttt{Step 3.} We show that the choice of the angles $\overline{\Theta }$ is nondegenerate, that is
\eqnb\label{h_is_1}
\overline{H} (\overline{\Theta }) = (1,1,\ldots , 1)
\eqne
and
\eqnb\label{naF_is_inver}
\nabla \overline{ F } (\overline{\Theta }) \text{ is an invertible matrix }
\eqne
for $M\in \{ 2,3,5,7,9\}$, where $\overline{\Theta }$ denotes the vector of possible angles given by \eqref{angles_possible}, see Section~\ref{sec_nondeg_theta} for details. 

\noindent\texttt{Step 4.}  We  prove the existence of $a_0 \in \R$ and a $C^1$ map $\Theta \colon [a_0,\infty ) \to \D $ such that 
\begin{equation*}
\left| \Theta(a) - \overline{\Theta } \right| \leq \varepsilon
\end{equation*}
and $F (a,\Theta (a))=0$ for all $a\geq a_0$. Note that this step proves that the solution $g'=G(a,\Theta (a))$ of \eqref{system_for_g'} is real for each $a\geq a_0$.

To this end, we note that \eqref{h_is_1} shows that $\widetilde{F}(0,\overline{\Theta } )=0$, and so noting \eqref{naF_is_inver} we can use the Implicit Function Theorem to obtain $t_0>0$ and a $C^{1}$ function $\widetilde\Theta:[0,t_{0})\to \R^{M-1}$ such that  
\[ 
\widetilde F(t, \widetilde\Theta(t)) = 0,\quad t\in(0,t_{0})
\] 
and $\widetilde\Theta(0) =\overline{\Theta }$. Letting $a_0\coloneqq 1/t_0$ and $\Theta(a)\coloneqq  \widetilde \Theta(1/a)$ we obtain the claim of this step.

Moreover, we obtain asymptotic expansions as $a\to \infty$, 
\begin{align}
\Theta(a) &= \overline{\Theta}   +\Theta_{-1} a^{-1} + o(a^{-1}), \label{expan1}\\
G(a,\Theta(a)) &=  \overline{G} (\overline{\Theta })  +  G_{-1} a^{-1} + o(a^{-1}), \label{expan2}
\end{align}
as  $a\to \infty$, where
\eqnb\label{def_of_-1}
\begin{split}
\Theta_{-1} &\coloneqq -\nabla \overline{F}(\overline{\Theta })^{-1} I (\overline{\Theta }),\\
G_{-1} &\coloneqq R (\overline{\Theta }) +  \nabla \overline{G}(\overline{\Theta } )\Theta_{-1},
\end{split}
\eqne
which we prove in Section~\ref{sec_expans}. We also note that $\overline{G}(\overline{\Theta }) = (1,\ldots , 1)$, a consequence of \eqref{h_is_1}.

\noindent\texttt{Step 5.} We show that if the  the nondegeneracy condition \eqref{naF_is_inver} holds then  \eqref{limit_at_angles} can be used to prove existence of unique $g_0 \in \R\setminus \{ 0 \}$ and $\mu \in \R$ such that \eqref{matrix_eq_main} holds for $\Theta = \Theta (a)$ and every $a\geq a_0$. (Note that this step proves Theorems~\ref{thm_main} and \ref{cor_suff_cond}, as required.)

The pair $(g_0,\mu )$ can be determined using \eqref{eq-g-n}, namely
\eqnb\label{eq-g-n_repeat}
-g_0 \frac{a}{\sinh (\pi A) }  \left( \cosh (\pi A) + \sum_{l=1}^{M-1} g_l'(a) \ee^{A(\theta_l (a)-\pi )} \right) = \mu\left(i - \frac{1}{a}\right) + \frac{1+a^{2}}{2a},
\eqne
where
\eqnb\label{def_of_gl'}
g_l'(a) = G(a,\Theta(a) ) = \frac{\H_l (a , \Theta (a))}{\K (a, \Theta (a) )},
\eqne
due to Step 1. Substituting \eqref{def_of_gl'} into \eqref{eq-g-n_repeat} and multiplying both sides of the resulting equation by $\K(a,\Theta (a))$ gives
\eqnb\label{eq-g-n_repeat1}
\begin{split}
&-\frac{1+a^2}{2a} \K (a, \Theta (a)) \\
&= g_0 \underbrace{\frac{ a }{\sinh (\pi A)} \left( \cosh (\pi A) \K (a, \Theta (a)) + \sum_{l=1}^{M-1} \H_l (a , \Theta (a) ) \ee^{A(\theta_l (a)-\pi )} \right)}_{=: E_1 (a) } + \mu \underbrace{\left(i-\frac1a\right) \K (a, \Theta (a) )}_{=: E_2 (a)}
\end{split}
\eqne

Thus it suffices to show that $E_1 (a)$, $E_2(a)$ are linearly independent (as vectors in the complex plane $\C$), which can be shown directly in the case $M=2$ (see Section~\ref{sec_caseM2}), while in the case of  odd $M\geq 3$ (Section~\ref{sec_case_Mgeq3}) it can be proved, provided  invertibility \eqref{naF_is_inver} of $\nabla \overline{F} (\overline{\Theta})$, by applying   \eqref{limit_at_angles}  to expand $E_1(a)$ and $E_2(a)$ at $a\to \infty$ and then recalling some properties of algebraic numbers as well as using the fact that $\pi$ is a transcendental number.\\

The structure of the article is as follows. In Section~\ref{sec_dets} we prove the formulas \eqref{toshow_dets} for the determinants of $\B$ and $\B_l$, $l=1,\ldots , M-1$. We then study the convergence properties as $a\to \infty$, as mentioned in Step~2 above, in Section~\ref{sec_conv_claims}. Section~\ref{sec_nondeg_theta} is devoted to study nondegeneracy properties of the angles \eqref{angles_possible}, as mentioned in Step~1 and Step~3 above, and we verify the asymptotic expansions \eqref{expan1}--\eqref{expan2} in Section~\ref{sec_expans}. Finally, we prove linear independence of $E_1(a),E_2(a) \in \C$, as mentioned in Step~5 above, in Section~\ref{sec_choice_g_mu}.

\section{The determinants}\label{sec_dets}

In this section we prove \eqref{toshow_dets}. Namely, we find $\det \B$ and $\det \B_l$, $l=1,\ldots , M-1$. To this end, we recall from \eqref{oA} and \eqref{def_r_rk} that 
\[
\B \coloneqq \begin{pmatrix}  
 \frac{cr_2-rr_1}{r_2} & r\frac{r_2-r_1}{r_3}  & r\frac{r_2-r_1}{r_4}  &\ldots &  r\frac{r_2-r_1}{r_M}  \\
 \frac{r^{-1}r_3-rr_1}{r_2} & \frac{cr_3 - rr_1 }{r_3} & r\frac{r_3-r_1}{r_4} & \ldots & r \frac{r_3-r_1}{r_M} \\
 \frac{r^{-1}r_4-rr_1}{r_2} & \frac{r^{-1} r_4 - rr_1 }{r_3} & \frac{cr_4-rr_1}{r_4} & \ldots & r \frac{r_4-r_1}{r_M} \\
\vdots & \vdots & \vdots &\ddots & \vdots \\
  \frac{r^{-1}r_M-rr_1}{r_2} & \frac{r^{-1}r_M-rr_1}{r_3}& \frac{r^{-1}r_M-rr_1}{r_4} &\ldots & \frac{cr_M-rr_1}{r_M}
\end{pmatrix},
\]
where $r\coloneqq \ee^{-A\pi }$, $r_k\coloneqq \ee^{-A\theta_{k-1}}$ for $k=1,\ldots ,M$, and we recall \eqref{def_Bl} for the definition of $\B_l$. We first find $\det \B$.

\begin{proposition}\label{prop_B}
If $M\ge 2$, $a>0$ and $\Theta=(\theta_{1},\ldots,\theta_{M-1})\in\R^{M-1}$, then 
\[
\det\mathcal{B}(a,\Theta)= \sinh^{M-2}(\pi A)\left( \frac{\ee^{\pi A} +(-1)^M\ee^{-\pi A }}{2} + (-1)^{M}\sum_{k=1}^{M-1}(-1)^{k}\ee^{A(\theta_{k}-\pi)} \right).
 \]
\end{proposition}
\begin{proof}
Note that
\[
\B \begin{pmatrix}  
1 & -\frac{r_2}{r_3} & -\frac{r_2}{r_4}   & \ldots & -\frac{r_2}{r_M} \\ 0 & 1 & 0 &\ldots & 0   \\
0   &0  &1  & \ldots & 0  \\
\vdots & \vdots & \vdots & \ddots & \vdots  \\
0 &0 & 0 & \ldots & 0  \\
0  & 0 & 0  &\ldots & 1 \end{pmatrix} =\begin{pmatrix}  
 \frac{cr_2-rr_1}{r_2} & s\frac{r_2}{r_3}  & s\frac{r_2}{r_4}  &\ldots &  s\frac{r_2}{r_M}  \\
 \frac{r^{-1}r_3-rr_1}{r_2} & s  & 2s\frac{r_3}{r_4} & \ldots &  2s\frac{r_3}{r_M} \\
 \frac{r^{-1}r_4-rr_1}{r_2} & 0 & s  & \ldots &  2s\frac{r_4}{r_M} \\
\vdots & \vdots & \vdots &\ddots & \vdots \\
  \frac{r^{-1}r_{M-1}-rr_1}{r_2} & 0 & 0  &\ldots & 2s\frac{r_{M-1}}{r_M} \\
  \frac{r^{-1}r_M-rr_1}{r_2} & 0 & 0  &\ldots & s 
\end{pmatrix}=:\B',
\]
where we recalled that $s\coloneqq (r-r^{-1})/2=-\sinh (\pi A)$. Moreover
\[
\B'\begin{pmatrix}  
1 & 0 & 0   & \ldots & 0 \\ 
0 & 1 & -\frac{r_3}{r_4} &\ldots & -\frac{r_3}{r_M}   \\
0  &0  &1  & \ldots & 0  \\
\vdots & \vdots & \vdots & \ddots & \vdots  \\
0 &0 & 0 & \ldots & 0  \\
0  & 0 & 0  &\ldots & 1 \end{pmatrix} =\begin{pmatrix}  
 \frac{cr_2-rr_1}{r_2} & s\frac{r_2}{r_3}  &0 &0 &\ldots & 0  \\
 \frac{r^{-1}r_3-rr_1}{r_2} & s  & s\frac{r_3}{r_4} & s\frac{r_3}{r_5}&\ldots &  s\frac{r_3}{r_M} \\
 \frac{r^{-1}r_4-rr_1}{r_2} & 0 & s  &2s\frac{r_4}{r_5}& \ldots &  2s\frac{ r_4}{r_M} \\
 \frac{r^{-1}r_5-rr_1}{r_2} & 0 & 0  &s& \ldots &  2s\frac{ r_5}{r_M} \\
\vdots & \vdots &\vdots & \vdots &\ddots & \vdots \\
  \frac{r^{-1}r_{M-1}-rr_1}{r_2} & 0 & 0 &0 &\ldots &2s \frac{r_{M-1}}{r_M} \\
  \frac{r^{-1}r_M-rr_1}{r_2} & 0 & 0 &0 &\ldots & s 
\end{pmatrix}=:\B''.
\]
Continuing we obtain
\begin{align*}
&\B''\begin{pmatrix}  
1&0 & 0 & 0   & \ldots & 0 \\ 
0&1 & 0 & 0   & \ldots & 0 \\ 
0&0 & 1 & -\frac{r_4}{r_5} &\ldots & -\frac{r_4}{r_M}   \\
0&0  &0  &1  & \ldots & 0  \\
\vdots &\vdots & \vdots & \vdots & \ddots & \vdots  \\
0&0 &0 & 0 & \ldots & 0  \\
0&0  & 0 & 0  &\ldots & 1 \end{pmatrix}\ldots \begin{pmatrix}  
1&0 & 0 & 0   & \ldots & 0 \\ 
0&1 & 0 & 0   & \ldots & 0 \\ 
0&0 & 1 & 0 &\ldots & 0   \\
0&0  &0  &1  & \ldots & 0  \\
\vdots &\vdots & \vdots & \vdots & \ddots & \vdots  \\
0&0 &0 & 0 & \ldots & -\frac{r_{M-1}}{r_M}  \\
0&0  & 0 & 0  &\ldots & 1 \end{pmatrix} \\[10pt]
&\qquad =\begin{pmatrix}  
 \frac{cr_2-rr_1}{r_2} & s\frac{r_2}{r_3}  &0 &0 &\ldots & 0  \\
 \frac{r^{-1}r_3-rr_1}{r_2} & s  & s\frac{r_3}{r_4} & 0&\ldots & 0 \\
 \frac{r^{-1}r_4-rr_1}{r_2} & 0 & s  &s\frac{r_4}{r_5}& \ldots &  0 \\
 \frac{r^{-1}r_5-rr_1}{r_2} & 0 & 0  &s& \ldots &  0\\
\vdots & \vdots &\vdots & \vdots &\ddots & \vdots \\
  \frac{r^{-1}r_{M-1}-rr_1}{r_2} & 0 & 0 &0 &\ldots &s \frac{r_{M-1}}{r_M} \\
  \frac{r^{-1}r_M-rr_1}{r_2} & 0 & 0 &0 &\ldots & s 
\end{pmatrix} =:\B'''.
\end{align*}
We will now write $\B\equiv \B_{M-1}$ and $\B'''\equiv \B_{M-1}'''$ to emphasize the size of the matrix. Using Laplace's expansion with respect to the last row we obtain
\[
\det \B_{M-1} = \det \B'''_{M-1} = (-1)^M s^{M-2} \left( r^{-1} - r \frac{r_1}{r_M} \right) + s \det \B'''_{M-2}.
\]
Thus, since
\[
\det \B'''_2 = \det \begin{pmatrix}
\frac{cr_2-rr_1}{r_2} & s\frac{r_2}{r_3} \\
\frac{r^{-1} r_3 - rr_1}{r_2} & s
\end{pmatrix} = s^2 + sr \left( \frac{r_1}{r_3} - \frac{r_1}{r_2} \right),
\]
we obtain
\eqnb\label{det_B_M-1}
\begin{split}\det \B_{M-1}''' &= s^{M-2} r\sum_{k=2}^M (-1)^{k+1} \frac{r_1}{r_k}  + \begin{cases} s^{M-2}c \qquad &M\text{ is even,}\\
s^{M-1} &M\text{ is odd}
\end{cases}\\
&= s^{M-2} \left(\frac{r+(-1)^M {r^{-1}}}{2} + \sum_{k=2}^M (-1)^{k+1} \frac{rr_1}{r_k} \right)
\end{split}
\eqne
for all $M\geq 3$, by induction.
\end{proof}
We now fix $l\in \{ 1, \ldots , M-1\}$ and we consider $\B_l$ (recall \eqref{def_Bl}).
\begin{proposition}\label{prop_Bl}
If $M\ge 2$ and $1\le l\le M-1$, then 
\begin{equation}
\begin{aligned}
\det\mathcal{B}_{l}(a,\Theta) & = \sinh^{M-2}(\pi A)\left(\frac{ \ee^{\pi A} - (-1)^M\ee^{-\pi A} }{2} \right)\\ 
&\quad \ +\sinh^{M-2}(\pi A)\left(\sum_{k=0}^{l-1}(-1)^{l+k}\ee^{A(\theta_{k}-\theta_{l}+\pi)} +(-1)^M \sum_{k=l}^{M-1}(-1)^{l+k}\ee^{A(\theta_{k}-\theta_{l}-\pi)}\right)
\end{aligned}
\end{equation}

for $a>0$ and $\Theta = (\theta_{1}, \ldots, \theta_{M-1})\in\R^{M-1}$. 
\end{proposition}
\begin{proof} We will consider the matrix $E_{M-1,l}=E_{M-1,l} (r, r_{1},r_{2},\ldots,r_{M})$ with the opposite sign of the $l$-th column, namely we set 
\[
E_{M-1,l} \coloneqq \left(  \begin{array}{ccccc|c|ccc}
 \frac{cr_2-rr_1}{r_2} & \frac{rr_2-rr_1}{r_3}  & \frac{rr_2-rr_1}{r_4}  &\ldots &\frac{rr_2-rr_1}{r_l}& \frac{r^{-1}r_2-cr_1}{r_1}&\frac{rr_2-rr_1}{r_{l+2}} & \ldots & \frac{rr_2-rr_1}{r_M}  \\
 \frac{r^{-1}r_3-rr_1}{r_2} & \frac{cr_3 - rr_1 }{r_3} & \frac{rr_3-rr_1}{r_4} & \ldots &\frac{rr_3-rr_1}{r_l}& \frac{r^{-1}r_3-cr_1}{r_1}&\frac{rr_3-rr_1}{r_{l+2}} & \ldots &  \frac{rr_3-rr_1}{r_M} \\
 \frac{r^{-1}r_4-rr_1}{r_2} & \frac{r^{-1} r_4 - rr_1 }{r_3} & \frac{cr_4-rr_1}{r_4} & \ldots &\frac{rr_4-rr_1}{r_l}& \frac{r^{-1}r_4-cr_1}{r_1}&\frac{rr_4-rr_1}{r_{l+2}}& \ldots &  \frac{rr_4-rr_1}{r_M} \\
 \vdots & \vdots & \vdots &\ddots & \vdots &\vdots &\vdots & \ddots & \vdots \\
 \frac{r^{-1}r_M-rr_1}{r_2} & \frac{r^{-1}r_M-rr_1}{r_3}& \frac{r^{-1}r_M-rr_1}{r_4} &\ldots &\frac{rr_M-rr_1}{r_l}& \frac{r^{-1}r_M-cr_1}{r_1}&\frac{r^{-1}r_M-rr_1}{r_{l+2}}& \ldots & \frac{cr_M-rr_1}{r_M}
 \end{array} \right),
\]
 where the vertical lines indicate the $l$-th column, and we will write 
\[
c\coloneqq \frac{r+r^{-1} }2 , \qquad s \coloneqq \frac{r-r^{-1}}2.
\]
Here the first index, ``$M-1$'', indicates the size of the matrix. We shall  write $E_{M-1} \coloneqq E_{M-1,l}$ for brevity if $l\in \{ 1 , \ldots , M-1 \} $ is fixed. Since $\det \B_l = -\det E_{M-1}$ in the particular choice 
\[r\coloneqq \ee^{-A\pi }, \quad r_k \coloneqq \ee^{-A\theta_{k-1}}, \quad \text{ in which case } s=  -\sinh (\pi A),\quad c= \cosh (\pi A),\]
we need to show that
\begin{equation}
\begin{aligned}\label{det_e_nts}
-\det E_{M-1} & = (-s)^{M-2}\left(\frac{r^{-1}-(-1)^M r}{2} \right)\\
&\quad +(-s)^{M-2}\left(\sum_{k=1}^{l} (-1)^{l+k+1}r^{-1} \frac{r_{l+1}}{r_{k}} + (-1)^M \sum_{k=l+1}^{M} (-1)^{l+k+1} r\frac{r_{l+1}}{r_k}\right)
\end{aligned}
\end{equation}
for $M\geq 3$ and $l\in \{ 1, \ldots , M-1 \}$. We begin with the case $2\le l \le M-2$. We define 
\begin{align*}
R_{1}:=\begin{pmatrix}  
1 & -\frac{r_2}{r_3} & -\frac{r_2}{r_4}   & \ldots & -\frac{r_2}{r_M} \\ 0 & 1 & 0 &\ldots & 0   \\
0   &0  &1  & \ldots & 0  \\
\vdots & \vdots & \vdots & \ddots & \vdots  \\
0  & 0 & 0  &\ldots & 1 \end{pmatrix},\qquad 
R_{l-1}:=\begin{pmatrix}  
1 & 0 &   \ldots & 0&0&0&\ldots & 0 \\ 0 & 1 & \ldots & 0&0&0&\ldots &0   \\
\vdots &  \vdots & \ddots & \vdots & \vdots &\vdots &\ddots & \vdots  \\
0   & 0  &\ldots & 1&0&-\frac{r_l}{r_{l+2}}&\ldots & -\frac{r_l}{r_M} \\
0   & 0  &\ldots & 0&1&0&\ldots & 0 \\
0   & 0  &\ldots & 0&0&1&\ldots & 0 \\
\vdots &  \vdots & \ddots & \vdots & \vdots &\vdots &\ddots & \vdots  \\
0   & 0  &\ldots & 0&0&0&\ldots & 1 \end{pmatrix}
\end{align*}
and, for any $2\le k \le l-2$, write
\[
R_{k}:=\begin{pmatrix}
1 & 0 & \ldots & 0&0&0&\ldots&0&0&0&\ldots & 0\\ 
0 & 1 & \ldots & 0& 0&0&\ldots&0&0&0&\ldots &0 \\
\vdots &  \vdots & \ddots & \vdots&\vdots &\vdots&\ldots&\vdots&\vdots & \vdots &\ldots & 0 \\
0   & 0  &\ldots & 1& -\frac{r_{k}}{r_{k+1}} & -\frac{r_{k}}{r_{k+2}} & \ldots&-\frac{r_{k}}{r_{l}} & 0 &-\frac{r_k}{r_{l+2}}&\ldots & -\frac{r_k}{r_M}\\
0   & 0  &\ldots & 0&1&0&\ldots&0&0&0& \ldots & 0\\
0   & 0  &\ldots & 0&0&1&\ldots&0&0&0& \ldots & 0\\
0   & 0  &\ldots & 0&0&0&\ldots&1&0&0&\ldots & 0\\
0   & 0  &\ldots & 0&0&0&\ldots&0&1&0&\ldots & 0\\
0   & 0  &\ldots & 0&0&0&\ldots&0&0&1&\ldots & 0\\
\vdots &  \vdots & \ddots & \vdots&\vdots & \vdots&\ldots &\vdots &\vdots&\vdots&\ddots & \vdots \\
0   & 0  &\ldots & 0&0&0&\ldots&0&0&0&\ldots & 1
\end{pmatrix}.\]

We perform the elimination procedure up to the $l$-th column to obtain
\[\begin{split}
E':=ER_{1}\ldots R_{l-1}&= \left(  \begin{array}{ccccc|c|ccc}
 \frac{cr_2-rr_1}{r_2} & s \frac{r_2}{r_3}  & 0   &\ldots &0 & -s\frac{r_2-r_1}{r_1} &0  & \ldots & 0  \\
 \frac{r^{-1}r_3-rr_1}{r_2} & s & s\frac{r_3}{r_4}  & \ldots &0 & s &0  & \ldots & 0  \\
 \frac{r^{-1}r_4-rr_1}{r_2} & 0 & s & \ldots & 0 & s & 0 & \ldots &  0  \\
 \vdots & \vdots & \vdots &\ddots & \vdots &\vdots &\vdots & \ddots & \vdots \\
 \frac{r^{-1}r_l-rr_1}{r_2} & 0 & 0 & \ldots & s & s & s\frac{r_l}{r_{l+2}} & \ldots &  s\frac{r_l}{r_{M}} \\
 \frac{r^{-1}r_{l+1}-rr_1}{r_2} & 0 & 0 & \ldots & 0 & s & 2s\frac{r_{l+1}}{r_{l+2}} & \ldots &  2s\frac{r_{l+1}}{r_{M}}  \\
 \frac{r^{-1}r_{l+2}-rr_1}{r_2} & 0 & 0 & \ldots & 0 & s & s & \ldots &  2s\frac{r_{l+2}}{r_{M}}  \\
 \frac{r^{-1}r_{l+3}-rr_1}{r_2} & 0 & 0 & \ldots & 0 & s & 0 & \ldots &  2s\frac{r_{l+3}}{r_{M}}  \\
 \vdots & \vdots & \vdots &\ddots & \vdots &\vdots &\vdots & \ddots & \vdots \\
 \frac{r^{-1}r_{M-1}-rr_1}{r_2} & 0& 0 &\ldots &0& s&0& \ldots & 2s \frac{r_{M-1}}{r_M}\\
 \frac{r^{-1}r_M-rr_1}{r_2} & 0& 0 &\ldots &0& s&0& \ldots & s
 \end{array} \right).
\end{split}
\]
We continue the elimination procedure past the $l$-th column, that is, we let
\[\begin{split}
E_{M-1}''&\coloneqq E' \left( \begin{array}{cccc|c|cccc}
1 & 0 &   \ldots & 0&0&0&0&\ldots & 0 \\ 0 & 1 & \ldots & 0&0&0&0&\ldots &0   \\
\vdots &  \vdots & \ddots & \vdots & \vdots &\vdots &0&\ldots & 0  \\
0   & 0  &\ldots & 1&0&0&0&\ldots & 0 \\
0   & 0  &\ldots & 0&1&0&0&\ldots & 0 \\
0   & 0  &\ldots & 0&0&1&-\frac{r_{l+2}}{r_{l+3}}&\ldots & -\frac{r_{l+2}}{r_{M}} \\
0   & 0  &\ldots & 0&0&0&1&\ldots & 0 \\
\vdots &  \vdots & \ddots & \vdots & \vdots &\vdots &0&\ldots & 0  \\
0   & 0  &\ldots & 0&0&0&0&\ldots & 1 \end{array}\right)\ldots  \begin{pmatrix}  
1 & 0 &  \ldots & 0&0 \\
0 & 1 & \ldots & 0&0   \\
\vdots & \vdots  & \ddots & \vdots &\vdots  \\
0 &0  & \ldots & 1&-\frac{r_{M-1}}{r_M}  \\
0  & 0   &\ldots & 0&1 \end{pmatrix} ,
\end{split}
\]
where the subindex ``$M-1$'' emphasizes the size of the matrix, to obtain that 
\begin{figure}[h]
\centering
\includegraphics[width=\textwidth]{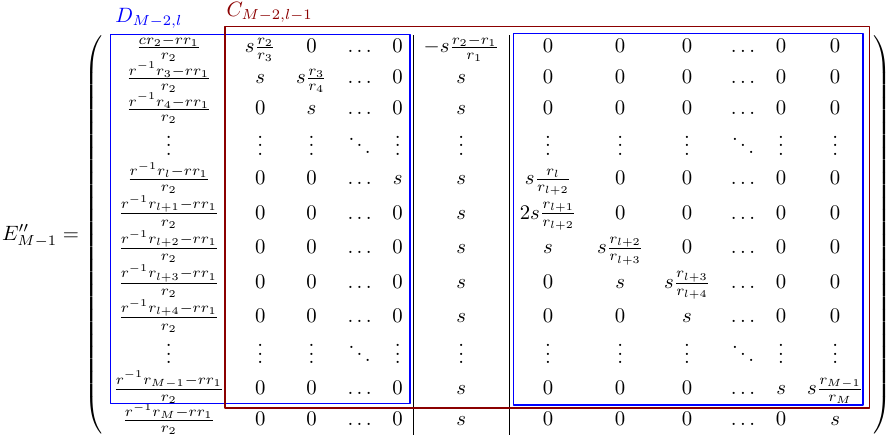}
\end{figure} 
 
Applying the Laplace expansion with respect to the last row, we infer that
\eqnb\label{det_of_e}
\det E_{M-1}'' = (-1)^M \frac{r^{-1}r_M -rr_1}{r_2}\det C_{M-2,l-1} +(-1)^{M+l+1} s \det D_{M-2,l}+ s \det E_{M-2}''
\eqne
for each $l\in \{ 2 , \ldots , M-2 \}$, where $C_{M-2,l-1}$ and $D_{M-2,l}$ are submatrices of $E_{M-1}''$ pointed out above. 

Note that Laplace's expansion with respect to the last column of $D_{M-2,l}$ gives
\eqnb\label{det_D}
\begin{split}
\det D_{M-2,l} &= s \frac{r_{M-1}}{r_M } \det D_{M-3,l} =\ldots = s^{M-2-l} \frac{r_{l+2}}{r_M } \det D_{l,l}\\
&= s^{M-2-l} \frac{r_{l+2}}{r_M }\left(  2s\frac{r_{l+1} }{r_{l+2}} \det \B_{l-1}''' - s\frac{r_l}{r_{l+2}} \cdot (-s)^{l-2} \frac{r^{-1}r_{l+1}-rr_1}{r_l}  \right)\\
& = s^{M-3} \frac{1}{r_M }\left(  {r_{l+1} } \left( {r}+(-1)^l {r^{-1}}+2 \sum_{k=2}^l (-1)^{k+1} \frac{rr_1}{r_k}   \right) +(-1)^{l+1} ({r^{-1}r_{l+1}-rr_1})  \right)\\
& = s^{M-3} \frac{1}{r_M }\left(  {r_{l+1} } \left( {r}+2 \sum_{k=2}^l (-1)^{k+1} \frac{rr_1}{r_k}   \right) +(-1)^{l}rr_1  \right),
\end{split}
\eqne
where, in the second line, we used the fact that
\[
\det \begin{pmatrix} 
\frac{cr_2-rr_1}{r_2} & s\frac{r_2}{r_3} & 0 &\ldots & 0&0 \\
\frac{r^{-1}r_3-rr_1}{r_2} & s  & s\frac{r_3}{r_4} &\ldots &0& 0 \\
\frac{r^{-1}r_4-rr_1}{r_2} & 0 & s &\ldots & 0&0 \\
\vdots & \vdots & \vdots & \ddots & \vdots & \vdots \\
\frac{r^{-1}r_{l-1}-rr_1}{r_2} & 0 & 0 &\ldots & s &s\frac{r_{l-1}}{r_l}  \\
\frac{r^{-1}r_{l+1}-rr_1}{r_2} & 0 & 0 &\ldots & 0 &0 
 \end{pmatrix} = (-s)^{l-2} \frac{r^{-1}r_{l+1}-rr_1}{r_l},
\]
and, in the third line, where we have recalled \eqref{det_B_M-1},
\eqnb\label{det_B_M-1_repeated}
\det \B_{l-1}''' = \frac{s^{l-2}}2 \left( {r}+(-1)^l {r^{-1}}+2 \sum_{k=2}^l (-1)^{k+1} \frac{rr_1}{r_k}   \right).
\eqne

Similarly, Laplace's expansion with respect to the last column of $C_{M-2,l-1}$ gives
\eqnb\label{det_C_temp}\begin{split}
\det C_{M-2,l-1} &= s\frac{r_{M-1}}{r_M } \det C_{M-3,l-1} = \ldots = s^{M-2-l} \frac{r_{l+2}}{r_M } \det C_{l,l-1}\\
&= s^{M-2-l} \frac{r_{l+2}}{r_M } \left( 2s\frac{r_{l+1}}{r_{l+2}} \det C_{l-1,l-1} -s\frac{r_l}{r_{l+2}} \cdot s^{l-1} \frac{r_2}{r_l} \right) ,
\end{split}
\eqne
where, in the second line, we used the fact that
\[
\det \begin{pmatrix} 
 s\frac{r_2}{r_3} & 0 &\ldots & 0&0 &-s\frac{r_2-r_1}{r_1} \\
 s  & s\frac{r_3}{r_4} &\ldots &0& 0 &s\\
  \vdots & \vdots & \ddots & \vdots & \vdots &\vdots  \\
 0 & 0 &\ldots & s &s\frac{r_{l-1}}{r_l}  &s\\
0 & 0 &\ldots & 0 &0 &s
 \end{pmatrix} = s^{l-1} \frac{r_2}{r_l}.
\]
As for 
\[
C_{l-1,l-1} = \begin{pmatrix} 
 s\frac{r_2}{r_3} & 0 &\ldots & 0&0 &-s\frac{r_2-r_1}{r_1} \\
 s  & s\frac{r_3}{r_4} &\ldots &0& 0 &s\\
  \vdots & \vdots & \ddots & \vdots & \vdots &\vdots  \\
 0 & 0 &\ldots & s &s\frac{r_{l-1}}{r_l}  &s\\
0 & 0 &\ldots & 0 &s &s
 \end{pmatrix},
\]
we note that $\det C_{2,2} = s^2 \left( \frac{r_2 }{r_3} + \frac{r_2-r_1}{r_1} \right)= s^2 r_2 \left( \frac{1}{r_3}-\frac{1}{r_2}+\frac{1}{r_1} \right)$, and since 
\[
\det C_{l-1,l-1} = s^{l-1} \frac{r_2}{r_{l}} - s \det C_{l-2,l-2}
\]
for each $l\geq 4$, we obtain
\eqnb\label{det_C11}
\det C_{l-1,l-1} = s^{l-1} r_2 \sum_{k=1}^l \frac{(-1)^{k+l}}{r_k},
\eqne 
by induction. Applying \eqref{det_C11} in \eqref{det_C_temp}, we infer that
\begin{equation}\label{det_C}
\begin{split}
\det  C_{M-2,l-1}&= s^{M-2-l} \frac{r_{l+2}}{r_M } \left( 2s\frac{r_{l+1}}{r_{l+2}} s^{l-1} r_2 \sum_{k=1}^l \frac{(-1)^{k+l}}{r_k} -s\frac{r_l}{r_{l+2}} \cdot s^{l-1} \frac{r_2}{r_l} \right) \\
&= s^{M-2} \frac{r_{2}}{r_M } \left( 2{r_{l+1}} \sum_{k=1}^l \frac{(-1)^{k+l}}{r_k} - 1 \right).
\end{split}
\end{equation}
Therefore, substituting \eqref{det_D} and \eqref{det_C} in \eqref{det_of_e} gives
\[\begin{split}
\det E_{M-1}'' -s \det E_{M-2}'' &= (-1)^M \frac{r^{-1}r_M -rr_1}{r_2}s^{M-2} \frac{r_{2}}{r_M } \left( 2{r_{l+1}} \sum_{k=1}^l \frac{(-1)^{k+l}}{r_k} - 1 \right) \\
& \qquad +(-1)^{M+l+1}  s^{M-2} \frac{1}{r_M }\left(  {r_{l+1} } \left( {r}+2 \sum_{k=2}^l (-1)^{k+1} \frac{rr_1}{r_k}   \right) +(-1)^{l}rr_1  \right)\\
&= (-1)^M s^{M-2} \left[ 2\left(r^{-1} -r\frac{r_1}{r_M} +r\frac{r_1}{r_M} \right) r_{l+1}\sum_{k=1}^l \frac{(-1)^{k+l}}{r_k} \right. \\
&\qquad \underbrace{+2 (-1)^l \frac{r_{l+1} r}{r_M}}_{\text{ from case }k=1} \underbrace{- r^{-1} + r \frac{r_1}{r_M}}_{\text{ remaining from }1\text{st line}}
\left. + (-1)^{l+1} \frac{rr_{l+1}}{r_M}-\frac{rr_{1}}{r_M} \right] \\
&=(-1)^M s^{M-2} \left[ 2r^{-1} r_{l+1}\sum_{k=1}^l \frac{(-1)^{k+l}}{r_k}+ (-1)^l \frac{r_{l+1} r}{r_M}- r^{-1}  \right].
\end{split}
\]
Furthermore, if $l=M-1$ we apply the Laplace expansion with respect to the last row of $E''_{M-1}$ to obtain
\[\begin{split}
\det E_{M-1}'' &= (-1)^M \frac{r^{-1}r_M -rr_1}{r_2}\det C_{M-2,M-2} + s \det \mathcal{B}_{M-2}'''\\
&= s^{M-2} \left( (-1)^M (r^{-1} r_M -rr_1 ) \sum_{k=1}^{M-1} \frac{(-1)^{k+M+1}}{r_k} + \frac{r}{2} + (-1)^{M-1}\frac{r^{-1}}{2} + \sum_{k=2}^{M-1} (-1)^{k+1} \frac{r r_1}{r_k} \right)\\
&= s^{M-2} \left( r^{-1} r_M \sum_{k=1}^{M-1} \frac{(-1)^{k+1}}{r_k} - \frac{r}{2} + (-1)^{M-1} \frac{r^{-1}}{2} \right)
\end{split}
\] 
for each $l\in \{ 2 , \ldots , M-2 \}$, where we have used \eqref{det_C11} and \eqref{det_B_M-1_repeated} in the second line.
The above two equalities together with the induction argument, prove \eqref{det_e_nts}, as required. 

In the case $l=M-1$, we proceed with the following elimination procedure to obtain
\[
E' := ER_{1}\ldots R_{l-2} = \begin{pmatrix} 
\frac{cr_2-rr_1}{r_2} & s\frac{r_2}{r_3} & 0 &\ldots & 0&0 & -s\frac{r_2-r_1}{r_1} \\
\frac{r^{-1}r_3-rr_1}{r_2} & s  & s\frac{r_3}{r_4} &\ldots &0& 0 & s \\
\frac{r^{-1}r_4-rr_1}{r_2} & 0 & s &\ldots & 0&0 & s\\
\vdots & \vdots & \vdots & \ddots & \vdots & \vdots & \vdots\\
\frac{r^{-1}r_{M-2}-rr_1}{r_2} & 0 & 0 &\ldots & s &s\frac{r_{M-2}}{r_{M-1}} & s \\
\frac{r^{-1}r_{M-1}-rr_1}{r_2} & 0 & 0 &\ldots & 0 &s & s \\
\frac{r^{-1}r_{M}-rr_1}{r_2} & 0 & 0 &\ldots & 0 &0 & s
\end{pmatrix}.
\]
Expanding the determinant with respect to the last row, and using \eqref{det_B_M-1} and \eqref{det_C11} we infer that
\begin{align*}
\det E' & = (-1)^{M}\frac{r^{-1}r_{M}-rr_1}{r_2} \det C_{M-2,M-2}+s\det\mathcal{B}'''_{M-2} \\
& = (-1)^{M}\frac{r^{-1}r_{M}-rr_1}{r_2}s^{M-2} r_2 \sum_{k=1}^{M-1} \frac{(-1)^{k+M-1}}{r_k} + s^{M-2} \left( \frac{r+(-1)^{M-1}r^{-1}}{2}+ \sum_{k=2}^{M-1} (-1)^{k+1} \frac{rr_1}{r_k} \right) \\
& = r^{-1}r_{M}s^{M-2}\sum_{k=1}^{M-1} \frac{(-1)^{k-1}}{r_k} - s^{M-2}r + s^{M-2} \left( \frac{r+(-1)^{M-1}r^{-1}}{2}\right) \\
& = (-s)^{M-2} \left( \frac{(-1)^{M}r - r^{-1}}{2} +r^{-1}r_{M}(-s)^{M-2}\sum_{k=1}^{M-1} \frac{(-1)^{n+k-1}}{r_k} - (-s)^{M-2}(-1)^{n}r\right),
\end{align*}
which gives \eqref{det_e_nts}, as desired. To prove the formula \eqref{det_e_nts} in the remaining case $l=1$, we define the matrix $\tilde E$ by $\tilde E[i,j] := E[M-i+1, M-j+1]$ for $1\le i,j\le M$, where $E[i,j]$ denotes the $(i,j)$-th entry of $E$. Clearly $\det \tilde E = \det E$ and it is not difficult to check that the following symmetry relation holds
\begin{align*}
\tilde E_{M-1,1}(r,r_{1},r_{2},\ldots,r_{M-1}) = E_{M-1,M-1}(r^{-1},r^{2}r_{1},r_{M},r_{M-1},\ldots,r_{M-k+2},\ldots r_{3},r_{2})
\end{align*}
for any $r,r_{1},r_{2},\ldots,r_{M}\in\C\setminus\{0\}$. Therefore applying the formula \eqref{det_e_nts} in the already proved case $l=M-1$, we obtain
\begin{align*}
\det E_{M-1,1} & = \det \tilde E_{M-1,M-1} = \det E_{M,M-1}(r^{-1},r^{2}r_{1},r_{M},r_{M-1},\ldots,r_{M-k+2},\ldots r_{3},r_{2}) \\
& = -s^{n-2}\left(\frac{r-(-1)^n r^{-1}}{2} 
+ \sum_{k=2}^{n-1} (-1)^{n+k}r \frac{r_{2}}{r_{n-k+2} } + (-1)^{n+1}r \frac{r_{2}}{r^{2}r_1 } + (-1)^n r^{-1}\right) \\
& = -s^{n-2}\left(\frac{r+(-1)^n r^{-1}}{2} 
+ \sum_{k=3}^{n} (-1)^{k}r \frac{r_{2}}{r_{k} } + (-1)^{n+1} \frac{r_{2}}{rr_1 }\right),
\end{align*}
which provides the formula \eqref{det_e_nts} in the case $l=1$ and completes the proof of the proposition. 
\end{proof}

\section{Limiting behaviour as $a\to \infty $}\label{sec_conv_claims}

In this section we prove the claims from the Step~2 regarding the limiting behaviour of $-a^{2}\partial_{a}H(a,\Theta)$ as $a\to\infty$, where $\Theta\in\mathcal{D}$. 

We first show that the limit \eqref{R+iI} is uniform on compact subsets of the set $\mathcal{D}$. By definition \eqref{def_of_H} of $H(a,\Theta )$, we have
\begin{align}\label{eq-h}
\partial_{a}H_{l}(a,\Theta) = \frac{\partial_{a}\H_{l} (a,\Theta) \K (a,\Theta) - \partial_{a}\K (a,\Theta)\H_{l}(a,\Theta)}{\K (a,\Theta)^{2}}
\end{align}
for each $1\le l\le M-1$. Thus, since
\begin{align*}
\partial_{a} \K (a,\Theta) & = \pi\partial_{a} A(a)\frac{\ee^{\pi A} - (-1)^{M}\ee^{-\pi A }}{2} + (-1)^{M}\partial_{a} A(a)\sum_{k=1}^{M-1}(-1)^{k}(\theta_{k}-\pi)\ee^{A(\theta_{k}-\pi)},\\
\partial_{a}\H_l (a,\Theta) & = \partial_{a} A(a) \pi\frac{ \ee^{\pi A} + (-1)^{M}\ee^{-\pi A} }{2} 
 +\partial_{a} A(a) \sum_{k=0}^{l-1}(-1)^{l+k}(\theta_{k}-\theta_{l}+\pi) \ee^{A(\theta_{k}-\theta_{l}+\pi)} \\
 &\hspace{2cm} + (-1)^{M}\partial_{a} A(a)\sum_{k=l}^{M-1}(-1)^{l+k}(\theta_{k}-\theta_{l}-\pi)\ee^{A(\theta_{k}-\theta_{l}-\pi)},
\end{align*}
and since $-a^{2} \frac{\d A}{\d a } (a)\to -2$ as $a\to\infty$ (recall \eqref{def_of_A} that $A=-2ai/(a+i)$), we obtain  
\begin{align}
&\begin{aligned}\label{eq-h-1}
\lim_{a\to\infty}\left[-a^{2}\partial_{a}\K (a,\Theta)\right] &= ((-1)^{M}-1)\pi -2(-1)^{M}\sum_{k=1}^{M-1}(-1)^{k}(\theta_{k}-\pi)\ee^{-2i\theta_{k}}, 
\end{aligned} \\
&\begin{aligned}\label{eq-h-2}
\lim_{a\to\infty}\left[-a^{2}\partial_{a}\H_{l}(a,\Theta)\right] & = -(1+(-1)^{M})\pi -2 \sum_{k=0}^{l-1}(-1)^{l+k}(\theta_{k}-\theta_{l}+\pi) \ee^{-2i(\theta_{k}-\theta_{l})} \\
&\qquad -2 (-1)^{M}\sum_{k=l}^{M-1}(-1)^{l+k}(\theta_{k}-\theta_{l}-\pi)\ee^{-2i(\theta_{k}-\theta_{l})},
\end{aligned}
\end{align}
and the limits are uniform for $\Theta$ from compact subsets of $\mathcal{D}$. In view of \eqref{eq-h} we see that 
$-a^{2}\partial_{a}H(a,\Theta)$ converges as $a\to \infty$, uniformly on compact subsets of $\D$, which proves \eqref{R+iI}. \\

We can now prove that $\widetilde{F}$, given by the formula \eqref{def_of_tildeF}, is a $C^1$ map on the set 
$[0,1]\times B(\overline\Theta,\varepsilon)$. 
Indeed, noting that $\p_t \widetilde{F} (t,\Theta )= -a^2 \p_a F(a,\Theta )$ for $t=1/a >0$, the uniform convergence \eqref{R+iI} proved above shows that $\p_t \widetilde{F}$ is continuous on $[0,1]\times B(\overline\Theta,\varepsilon)$. A similar computation as above shows that $\partial_{\theta_l}\widetilde F(t,\Theta)$ is continuous on the same set.

We proceed to verify \eqref{limit_at_angles}. For brevity, in the rest of this section, we will apply the convention 
\eqnb\label{conv}
 \overline{\Theta } \equiv \Theta = (\theta_1,\ldots , \theta_{M-1} )  ,
\eqne
where $\overline\Theta$ is given by \eqref{angles_possible}. Observe that, if $M\ge 3$ is an odd number, then by \eqref{eq-l-1b}, 
\begin{equation}
\begin{aligned}\label{rr3}
\overline{\K } ({\Theta })&=\sum_{k=1}^{M-1}(-1)^{k+1}\ee^{-2i\theta_{k}}
= -\frac{2i\sin(2\theta_{1})}{(1+\ee^{-2i\theta_{1}})(1+\ee^{2i\theta_{1}})},
\end{aligned}
\end{equation}
which shows that $\overline\Theta={\Theta } \in \D$. Moreover,
\begin{align}\label{r1}
\overline{\H}_{l}(\Theta)  = -\frac{2i\sin(2\theta_{1})}{(1+\ee^{-2i\theta_{1}})(1+\ee^{2i\theta_{1}})},\quad 1\le l \le M-1,
\end{align}
which can be verified by first calculating the finite geometric series,
\eqnb\label{rr2}
\sum_{k=0}^{l-1}(-1)^{l+k}\ee^{2i(\theta_{l}-\theta_{k})} = (-1)^{l}\ee^{2i\theta_{l}} \sum_{k=0}^{l-1}(-1)^{k}\ee^{-2i\theta_{k}}
= -\frac{1-(-1)^{l}\ee^{2i\theta_{l}} + \ee^{2i\theta_{1}} - (-1)^{l}\ee^{2i\theta_{l+1}}}{(1+\ee^{-2i\theta_{1}})(1+\ee^{2i\theta_{1}})} ,
\eqne
\eqnb\label{rr1}
\sum_{k=l}^{M-1}(-1)^{l+k}\ee^{2i(\theta_{l}-\theta_{k})}=\sum_{k=0}^{M-l-1}(-1)^{k}\ee^{-2i\theta_{k}}
=\frac{(-1)^{l}\ee^{2i\theta_{l+1}}+\ee^{2i\theta_{1}}+(-1)^{l}\ee^{2i\theta_{l}} + 1}{(1+\ee^{-2i\theta_{1}})(1+\ee^{2i\theta_{1}})}.
\eqne 
Thus, recalling the definition \eqref{eq-l-2b} of $\overline{\H }_l$, we obtain  
\begin{equation}
\begin{aligned}\label{eq33}
&(1+\ee^{-2i\theta_{1}})(1+\ee^{2i\theta_{1}})\overline{\H}_l (\Theta ) \\
&\qquad=(1+\ee^{-2i\theta_{1}})(1+\ee^{2i\theta_{1}})\left(1+\sum_{k=0}^{l-1}(-1)^{l+k}\ee^{2i(\theta_{l}-\theta_{k})} - \sum_{k=l}^{M-1}(-1)^{l+k}\ee^{2i(\theta_{l}-\theta_{k})}\right) \\
&\qquad= 2+\ee^{2i\theta_{1}}+\ee^{-2i\theta_{1}} -(1-(-1)^{l}\ee^{2i\theta_{l}} + \ee^{2i\theta_{1}} - (-1)^{l}\ee^{2i\theta_{l+1}}) \\
&\quad\qquad  -((-1)^{l}\ee^{2i\theta_{l+1}}+\ee^{2i\theta_{1}}+(-1)^{l}\ee^{2i\theta_{l}} + 1)\\
&\qquad =-\ee^{2i\theta_{1}} + \ee^{-2i\theta_{1}} = -2i\sin(2\theta_{1}),
\end{aligned}
\end{equation}
which shows \eqref{r1}, and also \eqref{h_is_1}, since the right-hand sides of \eqref{rr3} and \eqref{r1} are the same. \\

To verify \eqref{limit_at_angles} we apply \eqref{rr3}--\eqref{r1} in the quotient rule \eqref{eq-h} to get 
\[
\lim_{a\to\infty}\left[-a^{2}\partial_{a}H_{l}(a,\Theta)\right] = -\frac{(1+\ee^{-2i\theta_1 })(1+\ee^{2i\theta_1})}{2i \sin(2\theta_1 ) } \underbrace{\left(\lim_{a\to\infty}\left[-a^{2}\partial_{a} \H_l (a,\Theta)\right] -\lim_{a\to\infty}\left[-a^{2}\partial_{a} \K (a,\Theta)\right] \right)}_{=:J}
\]
The ingredients of $J$ can be computed explicitly using \eqref{eq-h-1}--\eqref{eq-h-2}. Namely,  \eqref{eq-h-1} gives that
\begin{equation}
\begin{aligned}\label{asy-k}
&\lim_{a\to\infty}\left[-a^{2}\partial_{a} \K (a,\Theta )\right] 
 = 2\sum_{k=1}^{M-1}(-1)^{k}\theta_{k}\ee^{-2i\theta_{k}} - 2\pi\sum_{k=0}^{M-1}(-1)^{k}\ee^{-2i\theta_{k}} \\
& \qquad  = \frac{2\pi}{M}\frac{(M-1)\ee^{-2i\theta_{1}} + M-\ee^{-2i\theta_{1}}}{(1+\ee^{-2i\theta_{1}})^{2}} 
-4\pi \frac{1}{\ee^{-2i\theta_{1}} + 1} 
  = -\frac{2\pi}{M}\frac{(M+2)\ee^{-2i\theta_{1}} + M}{(1+\ee^{-2i\theta_{1}})^{2}}.
\end{aligned}
\end{equation}
On the other hand, by \eqref{eq-h-2}, we have
\begin{align*}
& \lim_{a\to\infty}\left[-a^{2}\partial_{a}\H_{l}(a,\Theta)\right] \\
&\quad = -2 \pi -2 \sum_{k=0}^{l-1}(-1)^{l+k}(\theta_{k}-\theta_{l}+\pi) \ee^{-2i(\theta_{k}-\theta_{l})} +2 \sum_{k=l+1}^{M-1}(-1)^{l+k}(\theta_{k}-\theta_{l}-\pi)\ee^{-2i(\theta_{k}-\theta_{l})} \\
& \quad = -2 \sum_{k=0}^{l-1}(-1)^{l+k}(\theta_{k}-\theta_{l}) \ee^{-2i(\theta_{k}-\theta_{l})} +2 \sum_{k=l+1}^{M-1}(-1)^{l+k}(\theta_{k}-\theta_{l})\ee^{-2i(\theta_{k}-\theta_{l})}-2\pi \sum_{k=0}^{M-1}(-1)^{l+k}\ee^{-2i(\theta_{k}-\theta_{l})} \\
&\quad =:I_{1}+I_{2}+I_{3}. 
\end{align*}
Simple calculations show that \!\footnote{\,We recall the formula $\sum_{k=1}^{m-1}k b^{k} = \frac{(m-1)b^{m+1}-mb^{m}-b}{(1-b)^{2}}$ for $m\ge 2$ and $b\in\C\setminus\{1\}$.}
\begin{align*}
I_{1} & = \sum_{k=1}^{l}(-1)^{k}\theta_{k} \ee^{2i\theta_{k}} = \frac{\pi}{M}\frac{l(-1)^{l}\ee^{2il\theta_{1}} + (l+1)(-1)^{l}\ee^{2i(l-1)\theta_{1}}-\ee^{-2i\theta_{1}}}{(1+\ee^{-2i\theta_{1}})^{2}}, \\
I_{2} & = \sum_{k=1}^{M-l-1}(-1)^{k}\theta_{k}\ee^{-2i\theta_{k}} = \frac{\pi}{M} \frac{(M-l-1)(-1)^{l}\ee^{2i(l-1)\theta_{1}} + (M-l)(-1)^{l}\ee^{2il\theta_{1}}-\ee^{-2i\theta_{1}}}{(1+\ee^{-2i\theta_{1}})^{2}},\\
I_{3} & = -2(-1)^{l}\frac{\ee^{2i(l-1)\theta_{1}} + \ee^{2il\theta_{1}}}{(1+\ee^{-2i\theta_{1}})^{2}}.
\end{align*}
Altogether,
\begin{align*}
I_{1}+I_{2}+I_{3}
 = -\frac{2\pi}{M}\frac{M(-1)^{l}\ee^{2i(l-1)\theta_{1}}+M(-1)^{l}\ee^{2il\theta_{1}}+2\ee^{-2i\theta_{1}}}{(1+\ee^{-2i\theta_{1}})^{2}}
\end{align*}
and so
\begin{align*}
&J = -\frac{2\pi}{M}\frac{M(-1)^{l}\ee^{2i(l-1)\theta_{1}}+M(-1)^{l}\ee^{2il\theta_{1}}+2\ee^{-2i\theta_{1}}}{(1+\ee^{-2i\theta_{1}})^{2}}
+ \frac{2\pi}{M}\frac{(M+2)\ee^{-2i\theta_{1}} + M}{(1+\ee^{-2i\theta_{1}})^{2}} \\
&\qquad = -2\pi\frac{(-1)^{l}\ee^{2il\theta_{1}}-1}{1+\ee^{-2i\theta_{1}}},
\end{align*}
that is 
\begin{align*}
\lim_{a\to\infty}\left[-a^{2}\partial_{a}H_{l}(a,\Theta)\right] = -i\pi \frac{((-1)^{l}\ee^{2il\theta_{1}}-1)(1+\ee^{2i\theta_{1}})}{\sin(2\theta_{1})},
\end{align*}
which is \eqref{limit_at_angles}, as required.

\section{Nondegeneracy of the point $\overline{\Theta}$}\label{sec_nondeg_theta}

In this section we verify some properties of the choice of angles $\overline{\Theta } = \frac{n\pi}{M}(1,2,\ldots ,M-1 ) $ (recall \eqref{angles_possible}) which were mentioned in Step 3. For simplicity, we continue the notation convention \eqref{conv} from the preceding section, i.e. write
\[
 \overline{\Theta } \equiv \Theta = (\theta_1,\ldots , \theta_{M-1} ).
\]
\subsection{Case $M=2$}\label{sec_case2}
Here we consider the case $M=2$. In this case $\overline{\Theta }=\Theta = {\theta}_1 = \pi/2$ and \eqref{eq-l-1b}--\eqref{eq-l-2b} give
\[
\overline{\K } ({\Theta } ) = 1+ (-1)\cdot \ee^{-2i \theta_1 } =2, \qquad \overline{\H } ({\Theta }) = \overline{\H }_1 ({\Theta }) = (-1)\ee^{2i ({\theta}_1- {\theta }_0 ) }+1 =2
\]
In particular ${\Theta } \in \D$, as mentioned in Step 1, and $\overline{H}({\Theta }) =\overline{\H }_1 ({\Theta })/  \overline{\K } ({\Theta } ) =1$, as claimed in \eqref{h_is_1}. It remains to show \eqref{naF_is_inver} and \eqref{limit_at_angles}. The former follows by noting that
\[
\overline{H }' (\Theta ) = -2i \ee^{2i\theta_1 } =2i
\]
and so, taking imaginary part, we obtain $ \overline{F }'  ( \Theta )= 2 \ne 0$, which proves \eqref{naF_is_inver}. 

\subsection{The case of odd $M\geq 3$}\label{sec_case_mleq9}
Here we show that ${\Theta }\in \D$ and we verify \eqref{h_is_1} and \eqref{naF_is_inver} in the case of odd $M\geq 3$. Recall that in this section, for the sake of brevity, we continue to use the convention \eqref{conv}, as we are only concerned with nondegeneracy at $\Theta = \overline{\Theta }$.

Let us observe that \eqref{h_is_1} is an immediate consequence of \eqref{rr3} and \eqref{r1}.

In order to show \eqref{naF_is_inver} we first calculate explicitly the entries of $\nabla  \overline{F} (\Theta )$ and $\nabla \overline{G} (\Theta )$.
\begin{lemma} For every $l,m\in \{ 1, \ldots , M-1 \}$
\eqnb\label{grad_Finfty}
\sin^{2}(\theta_{1})[\nabla \overline{F } (\Theta)]_{lm}=\left\{\begin{aligned}
& 2\sin^{2}(\theta_{1}) - (-1)^{m}\sin(2\theta_{1})\sin(2\theta_{m}) && \text{ if } \quad l=m,\\[5pt]
& (-1)^{m+1}\sin(2\theta_{1})(\sin(2\theta_{m}) + (-1)^{l}\sin(2\theta_{l-m})) && \text{ if } \quad l<m,\\[5pt]
& (-1)^{m+1}\sin(2\theta_{1})(\sin(2\theta_{m}) - (-1)^{l}\sin(2\theta_{l-m})) && \text{ if } \quad m<l
\end{aligned}\right.
\eqne
and 
\eqnb\label{grad_Ginfty}
\sin^{2}(\theta_{1})[\nabla \overline{G} (\Theta)]_{lm}=\left\{\begin{aligned}
& (-1)^{m}\cos(2\theta_{m})\sin(2\theta_{1})  && \text{ if } \quad l=m,\\[5pt]
& (-1)^{m}\sin(2\theta_{1})(\cos(2\theta_{m}) - (-1)^{l}\cos(2\theta_{l-m})) && \text{ if } \quad l<m,\\[5pt]
& (-1)^{m}\sin(2\theta_{1})(\cos(2\theta_{m}) + (-1)^{l}\cos(2\theta_{l-m})) && \text{ if } \quad m<l.
\end{aligned}\right.
\eqne
\end{lemma}
\begin{proof}
We consider the following three cases. \\
{\bf Case 1.} If we assume that $1\le l\le M-1$ and $l+1\le m\le M-1$, then
\begin{align*}
&\left(\sum_{k=1}^{M-1}(-1)^{k+1}\ee^{-2i\theta_{k}}\right)^{2}\partial_{m} \overline{H}_{l}(\Theta) = 2i(-1)^{l+m}\ee^{2i(\theta_{l}-\theta_{m})}\left(\sum_{k=1}^{M-1}(-1)^{k+1} \ee^{-2i\theta_{k}}\right)\\
&\qquad +2i(-1)^{m+1}\ee^{-2i\theta_{m}}\left(1+\sum_{k=0}^{l-1}(-1)^{l+k}\ee^{2i(\theta_{l}-\theta_{k})} - \sum_{k=l}^{M-1}(-1)^{l+k}\ee^{2i(\theta_{l}-\theta_{k})}\right)\\
&=2i(-1)^{m+1}\ee^{-2i\theta_{m}}-2i(-1)^{l+m+1}\ee^{2i(\theta_{l}-\theta_{m})}+4i\sum_{k=0}^{l-1}(-1)^{l+k+m+1}\ee^{2i(\theta_{l}-\theta_{k}-\theta_{m})}.
\end{align*}
In view of \eqref{rr2}, we have
\begin{align*}
&\sum_{k=0}^{l-1}(-1)^{l+k+m+1}\ee^{2i(\theta_{l}-\theta_{k}-\theta_{m})} = -(-1)^{m+1}\ee^{-2i\theta_{m}}\frac{1-(-1)^{l}\ee^{2i\theta_{l}} + \ee^{2i\theta_{1}} - (-1)^{l}\ee^{2i\theta_{l+1}}}{(1+\ee^{-2i\theta_{1}})(1+\ee^{2i\theta_{1}})}
\end{align*}
which in turn implies that
\begin{equation}
\begin{aligned}\label{eq22}
&\frac{1}{2}(1+\ee^{-2i\theta_{1}})(1+\ee^{2i\theta_{1}})\left(\sum_{k=1}^{M-1}(-1)^{k+1}\ee^{-2i\theta_{k}}\right)^{2}\partial_{m} \overline{H}_{l}(\Theta)\\ 
&\ \ =i(2+\ee^{-2i\theta_{1}} + \ee^{2i\theta_{1}})\left((-1)^{m+1}\ee^{-2i\theta_{m}}-(-1)^{l+m+1}\ee^{2i\theta_{l-m}}\right) \\
&\ \ \qquad -2i(-1)^{m+1}\ee^{-2i\theta_{m}}\left(1-(-1)^{l}\ee^{2i\theta_{l}} + \ee^{2i\theta_{1}} - (-1)^{l}\ee^{2i\theta_{l+1}}\right) \\
& \ \ = i(-1)^{m+1}\ee^{-2i\theta_{m+1}}-i(-1)^{l+m+1}\ee^{2i\theta_{l-m-1}}-i(-1)^{m+1}\ee^{-2i\theta_{m-1}}+i(-1)^{l+m+1}\ee^{2i\theta_{l-m+1}}.
\end{aligned}
\end{equation}
On the other hand, by \eqref{rr3} we have
\begin{equation}
\begin{aligned}\label{eq_11}
\frac{1}{2}(1+\ee^{-2i\theta_{1}})(1+\ee^{2i\theta_{1}})\left(\sum_{k=1}^{M-1}(-1)^{k+1}\ee^{-2i\theta_{k}}\right)^{2} = - 2\sin^{2}(\theta_{1}).
\end{aligned}
\end{equation}
By \eqref{eq_11}, the prefactor of $\partial_{m} \overline{H}_{l}(\Theta)$ in  \eqref{eq22} is a real number. Therefore, using \eqref{eq_11} and taking the imaginary part of  \eqref{eq22} gives
\begin{align*}
&-2\sin^{2}(\theta_{1})\partial_{m} \overline{F}_{l}(\Theta) = -2\sin^{2}(\theta_{1})\mathrm{Im}\,\partial_{m} \overline{H}_{l}(\Theta) \\
&\qquad=(-1)^{m+1}(\cos(2\theta_{m+1})-\cos(2\theta_{m-1}))-(-1)^{l+m+1}(\cos(2\theta_{l-m-1})-\cos(2\theta_{l-m+1})) \\
&\qquad =-2(-1)^{m+1}\sin(2\theta_{m})\sin(2\theta_{1}) - 2(-1)^{l+m+1}\sin(2\theta_{l-m})\sin(2\theta_{1}) \\
&\qquad =-2(-1)^{m+1}\sin(2\theta_{1})(\sin(2\theta_{m}) + (-1)^{l}\sin(2\theta_{l-m}))
\end{align*}
and furthermore
\begin{align*}
&-2\sin^{2}(\theta_{1})\partial_{m} \overline{G}_{l}(\Theta) = -2\sin^{2}(\theta_{1})\mathrm{Re}\,\partial_{m} \overline{H}_{l}(\Theta) \\
&\qquad= (-1)^{m+1}(\sin(2\theta_{m+1})-\sin(2\theta_{m-1}))+(-1)^{l+m+1}(\sin(2\theta_{l-m-1})-\sin(2\theta_{l-m+1})) \\
&\qquad= 2(-1)^{m+1}\sin(2\theta_{1})\cos(2\theta_{m}) - 2(-1)^{l+m+1}\sin(2\theta_{1})\cos(2\theta_{l-m}) \\
&\qquad= 2(-1)^{m+1}\sin(2\theta_{1})(\cos(2\theta_{m}) - (-1)^{l}\cos(2\theta_{l-m})),
\end{align*}
as desired. 

{\bf Case 2.} If we assume that $1\le l\le M-1$ and $1\le m\le l-1$, then 
\begin{align*}
&\left(\sum_{k=1}^{M-1}(-1)^{k+1}\ee^{-2i\theta_{k}}\right)^{2}\partial_{m} \overline{H}_{l}(\Theta)  = -2i(-1)^{l+m}\ee^{2i(\theta_{l}-\theta_{m})}\left(\sum_{k=1}^{M-1}(-1)^{k+1}\ee^{-2i\theta_{k}}\right) \\
&\quad\qquad + 2i(-1)^{m+1}\ee^{-2i\theta_{m}}\left(1+\sum_{k=0}^{l-1}(-1)^{l+k}\ee^{2i(\theta_{l}-\theta_{k})} - \sum_{k=l}^{M-1}(-1)^{l+k}\ee^{2i(\theta_{l}-\theta_{k})}\right) \\
&\quad = 2i(-1)^{m+1}\ee^{-2i\theta_{m}}+2i(-1)^{l+m+1}\ee^{2i(\theta_{l}-\theta_{m})}-4i\sum_{k=l}^{M-1}(-1)^{l+k+m+1}\ee^{2i(\theta_{l}-\theta_{k}-\theta_{m})}.
\end{align*}
Observe that the formula \eqref{rr1} gives
\begin{align*}
&\sum_{k=l}^{M-1}(-1)^{l+k+m+1}\ee^{2i(\theta_{l}-\theta_{k}-\theta_{m})} = (-1)^{m+1}\ee^{-2i\theta_{m}}\frac{(-1)^{l}\ee^{2i\theta_{l+1}}+\ee^{2i\theta_{1}}+(-1)^{l}\ee^{2i\theta_{l}} + 1}{(1+\ee^{-2i\theta_{1}})(1+\ee^{2i\theta_{1}})}
\end{align*}
and consequently we have
\begin{equation}
\begin{aligned}\label{eeq-11}
&\frac{1}{2}(1+\ee^{-2i\theta_{1}})(1+\ee^{2i\theta_{1}})\left(\sum_{k=1}^{M-1}(-1)^{k+1}\ee^{-2i\theta_{k}}\right)^{2}\partial_{m} \overline{H}_{l}(\Theta)\\ 
&\ =i\left((-1)^{m+1}\ee^{-2i\theta_{m}}+(-1)^{l+m+1}\ee^{2i\theta_{l-m}}\right)\left(2+\ee^{-2i\theta_{1}}+\ee^{2i\theta_{1}}\right) \\
&\quad -2i\left((-1)^{m+l+1}\ee^{2i\theta_{l-m+1}} + (-1)^{m+1}\ee^{-2i\theta_{m-1}} + (-1)^{m+l+1}\ee^{2i\theta_{l-m}} + (-1)^{m+1}\ee^{-2i\theta_{m}}\right)\\
& \ =i(-1)^{m+1}\ee^{-2i\theta_{m+1}} + i(-1)^{l+m+1}\ee^{2i\theta_{l-m-1}}-i(-1)^{m+1}\ee^{-2i\theta_{m-1}}-i(-1)^{m+l+1}\ee^{2i\theta_{l-m+1}}. 
\end{aligned}
\end{equation}
Similarly as in Case 1, the prefactor for $\partial_{m} \overline{H}_{l}(\Theta)$ in  \eqref{eeq-11} is a real number. Therefore, using \eqref{eq_11} and taking the imaginary part of  \eqref{eeq-11} gives
\begin{align*}
&-2\sin^{2}(\theta_{1})\partial_{m}\overline{F}_{l}(\Theta) = -2\sin^{2}(\theta_{1})\mathrm{Im}\,\partial_{m}\overline{H}_{l}(\Theta)\\ 
&\qquad=(-1)^{m+1}(\cos(2\theta_{m+1})-\cos(2\theta_{m-1})) + (-1)^{l+m+1}(\cos(2\theta_{l-m-1})-\cos(2\theta_{l-m+1}))\\
&\qquad = -2(-1)^{m+1}\sin(2\theta_{m})\sin(2\theta_{1}) + 2(-1)^{l+m+1}\sin(2\theta_{l-m})\sin(2\theta_{1})\\
&\qquad = -2(-1)^{m+1}\sin(2\theta_{1})(\sin(2\theta_{m}) - (-1)^{l}\sin(2\theta_{l-m}))
\end{align*}
and furthermore
\begin{align*}
&-2\sin^{2}(\theta_{1})\partial_{m}\overline{G}_{l}(\Theta) = -2\sin^{2}(\theta_{1})\mathrm{Re}\,\partial_{m}\overline{H}_{l}(\Theta)\\ 
& \qquad = (-1)^{m+1}(\sin(2\theta_{m+1})-\sin(2\theta_{m-1})) - (-1)^{l+m+1}(\sin(2\theta_{l-m-1})-\sin(2\theta_{l-m+1})) \\
& \qquad = 2(-1)^{m+1}\sin(2\theta_{1})\cos(2\theta_{m}) + 2(-1)^{l+m+1}\sin(2\theta_{1})\cos(2\theta_{l-m}) \\
& \qquad = 2(-1)^{m+1}\sin(2\theta_{1})(\cos(2\theta_{m}) + (-1)^{l}\cos(2\theta_{l-m})),
\end{align*}
as desired. 

{\bf Case 3.} Let us assume that $1\le l\le M-1$ and $m=l$. Then we have
\begin{align*}
&\left(\sum_{k=1}^{M-1}(-1)^{k+1}\ee^{-2i\theta_{k}}\right)^{2}\partial_{m}\overline{H}_{m}(\Theta) \\
&\qquad = 2i\left(\sum_{k=0}^{m-1}(-1)^{m+k}\ee^{2i(\theta_{m}-\theta_{k})} - \sum_{k=m+1}^{M-1}(-1)^{m+k}\ee^{2i(\theta_{m}-\theta_{k})}\right)\left(\sum_{k=1}^{M-1}(-1)^{k+1}\ee^{-2i\theta_{k}}\right) \\
&\qquad\quad + 2i(-1)^{m+1}\ee^{-2i\theta_{m}}\left(\sum_{k=0}^{m-1}(-1)^{m+k}\ee^{2i(\theta_{m}-\theta_{k})} - \sum_{k=m+1}^{M-1}(-1)^{m+k}\ee^{2i(\theta_{m}-\theta_{k})}\right).
\end{align*}
In view of the formulas \eqref{rr3} and \eqref{eq33} it follows that
\begin{align*}
&\frac{1}{2}(1+\ee^{-2i\theta_{1}})(1+\ee^{2i\theta_{1}})\left(\sum_{k=1}^{M-1}(-1)^{k+1}\ee^{-2i\theta_{k}}\right)^{2}\partial_{m} \overline{H}_{m}(\Theta)\\ 
&\qquad =-\frac{4i\sin^{2}(2\theta_{1})}{\left(1+\ee^{-2i\theta_{1}}\right)\left(1+\ee^{2i\theta_{1}}\right)} + 2(-1)^{m+1}\ee^{-2i\theta_{m}}\sin(2\theta_{1}),
\end{align*}
which in turn, by \eqref{eq_11}, gives 
\begin{align*}
-2\sin^{2}(\theta_{1})\partial_{m} \overline{F}_{m}(\Theta) & = -2\sin^{2}(\theta_{1})\mathrm{Im}\,\partial_{m} \overline{H}_{m}(\Theta) \\
& =-\frac{4\sin^{2}(2\theta_{1})}{\left(1+\ee^{-2i\theta_{1}}\right)\left(1+\ee^{2i\theta_{1}}\right)} + 2(-1)^{m}\sin(2\theta_{m})\sin(2\theta_{1}) \\
&=-4\sin^{2}(\theta_{1})+ 2(-1)^{m}\sin(2\theta_{m})\sin(2\theta_{1})
\end{align*}
and furthermore
\begin{align*}
-2\sin^{2}(\theta_{1})\partial_{m} \overline{F}_{m}(\Theta)  = -2\sin^{2}(\theta_{1})\mathrm{Re}\,\partial_{m} \overline{H}_{m}(\Theta) 
 =2(-1)^{m+1}\cos(2\theta_{m})\sin(2\theta_{1}),
\end{align*}
as desired. 
\end{proof}
The formulas  \eqref{grad_Finfty}--\eqref{grad_Ginfty} obtained in the above lemma can now be used to verify that $\det \nabla \overline{F}(\Theta )\neq 0$ for odd $M\geq 3$ (recall the convention \eqref{conv}).  As mentioned in the introduction we verify it only for $M\in \{ 3,5,7,9\}$. We have   
\begin{align*}
\det\left[\sin^{2}(\theta_{1})\nabla\overline{F}( \Theta)\right] = \lambda_{1}\lambda_{2}\ldots\lambda_{M-1},
\end{align*}
where $\lambda_{1},\lambda_{2},\ldots,\lambda_{M-1}$ are the eigenvalues of $\sin^{2}(\theta_{1})\nabla\overline{F}(\Theta)$, which can be computed explicitly, as listed in Table~\ref{table_eigenvalues} below.

\begin{center}
\begin{tabular}{ |>{\centering}m{1cm}|>{\centering}m{1cm}|m{12cm}| } 
  \hline 
 $M$ & $n$ & \vspace{2pt} Eigenvalues $\lambda_{1},\lambda_{2},\ldots,\lambda_{M-1}$ of the matrix $\sin^{2}(\theta_{1})\nabla\overline{F}(\Theta)$ \\[2pt]
 \hline
 $3$ & $1$, $2$ &  \vspace{2pt} $\lambda_{1} = \lambda_{2} = 9/4$  \\[2pt]
 \hline 
 \multirow{3}{*}{$5$} & $1$&  \vspace{2pt} $\lambda_{1} = \lambda_{2} = 5/2$ \newline $\lambda_{3} = \lambda_{4} = -5\sqrt{5}/8+5/8$ \\[2pt]
 \cline{2-3}
& $2$ & \vspace{2pt} $\lambda_{1} = \lambda_{2} = 5/2$ \newline $\lambda_{3} = \lambda_{4} = 5\sqrt{5}/8+5/8$ \\[2pt]
 \hline 
 \multirow{3}{*}{$7$} & $1$ &  \vspace{2pt} $\lambda_{1} = \lambda_{2} = -2 \cos^{2}(\pi/7)+3\cos ( \pi/7  ) +3/2$ \newline
$\lambda_{3} = \lambda_{4} = -4 \cos^{2}(\pi/7)-\cos ( \pi/7  ) +3$ \newline
$\lambda_{5} = \lambda_{6} = - \cos^{2}( \pi/7  )-2\cos ( \pi/7  ) +5/2$  \\[2pt]
 \cline{2-3}
& $2$ & \vspace{2pt} $\lambda_{1} = \lambda_{2} = 5 \cos^{2}(\pi/7) - \cos ( \pi/7  )/2 -1/4$ \newline
$\lambda_{3} = \lambda_{4} = 6\cos^{2}(\pi/7)-2\cos ( \pi/7  ) - 1$ \newline
$\lambda_{5} = \lambda_{6} = - 4\cos^{2}( \pi/7  )-\cos ( \pi/7  ) + 3$   \\[2pt]
\hline
 \multirow{5}{*}{$9$} & $1$ &  \vspace{2pt} $\lambda_{1} = \lambda_{2}  = -2\cos^{2}(\pi/9)+2\cos(\pi/9) +5/2$ \newline
$\lambda_{3} = \lambda_{4}  = -2 \cos^{2}(\pi/9)-\cos(\pi/9) +5/2$ \newline
$\lambda_{5} = \lambda_{6}  = 3/2-3\cos(\pi/9)$ \newline
$\lambda_{7} = \lambda_{8} = -5 \cos^{2}(\pi/9)+2\cos(\pi/9) +5/2$ \\[2pt]
 \cline{2-3}
 & $2$&  \vspace{2pt} $\lambda_{1} = \lambda_{2}  = 4\cos^{2}(\pi/9)-\cos(\pi/9) - 1/2$ \newline
$\lambda_{3} = \lambda_{4}  = 6\cos^{2}(\pi/9)- 3/2$ \newline
$\lambda_{5} = \lambda_{6}  = \cos^{2}(\pi/9) - 5\cos(\pi/9)/2 -1/2$ \newline
$\lambda_{7} = \lambda_{8} = -2\cos^{2}(\pi/9)-\cos(\pi/9) +5/2$  \\[2pt]
 \hline 
\end{tabular}
 \nopagebreak\vspace{5pt}
 \captionof{table}{Exact formulas for the eigenvalues of the matrix $\sin^{2}(\theta_{1})\nabla\overline{F}(\Theta)$ for $M=3,5,7,9$ (recall \eqref{conv}).}\label{table_eigenvalues} 
\end{center}

Since none of the eigenvalues equals $0$ we see that $\det\left[\sin^{2}(\theta_{1})\nabla\overline{F}(\Theta)\right] \ne 0$, as required.\\

\subsection{The case of even $M\geq 4$}\label{sec_m_geq4}

Here we verify Remark~\ref{rem_m_geq4}, that is we comment on the reason why Theorem~\ref{thm_main} does not apply to any even $M\geq 4$. \\

In fact, let us suppose that $M \ge 4$ is even. Then for any $\Theta = (\theta_1, \ldots , \theta_{M-1} ) \in \R^{M-1}$ we have that $
\K (\Theta ) = \sum_{k=0}^{M-1} (-1)^k \ee^{-2i\theta_k }$, by \eqref{eq-l-1b}. Thus for $\overline{\Theta } = \overline{\theta_1} (1,2,\ldots , M-1 )$ we have $\K (\overline{\Theta }) = \sum_{k=0}^{M-1} ( - \ee^{-2i\overline{\theta}_1} )^k=0$ if $\overline{\theta}_1=n\pi /M$ (recall $n\in\{1,2\}$).  This shows that the admissibility condition \eqref{admissible} fails, and so one cannot apply the argument of Section~\ref{sec_sketch} to obtain any spirals with such $\overline{\Theta }$. 

Moreover, we note that 
\begin{equation*}
\begin{aligned}
\overline{F}_{l}(\Theta)&=\mathrm{Im}\,(\overline{\H}_{l}(\Theta)/\overline{\K} (\Theta)) = \mathrm{Im}\left(\frac{1 + \sum_{k=0}^{l-1}(-1)^{l+k}\ee^{2 i(\theta_{l}-\theta_{k})} + \sum_{k=l+1}^{M-1}(-1)^{l+k}\ee^{2 i(\theta_{l}-\theta_{k})}}{1 + \sum_{k=1}^{M-1}(-1)^{k}\ee^{-2 i\theta_{k}}} \right)\\
&=\mathrm{Im}\left((-1)^{l}\ee^{2 i\theta_{l}}\right) = (-1)^{l}\sin(2\theta_{l})
\end{aligned}
\end{equation*}
for each $l\in \{ 1,\ldots , M-1\}$ and any $\Theta\in\mathcal{D}$, 
where we recalled \eqref{eq-l-1b}--\eqref{eq-l-2b} in the first line. 
Thus $\overline{F} (\Theta)=0$ holds if and only if $2\theta_l \in \{ k \pi \colon k \in \Z \}$ for every $l\in \{ 1, \ldots , M-1 \}$. Due to the ordering $0< \theta_1 < \ldots < \theta_{M-1} < 2\pi$ we therefore must have that $M=4$ and $\theta_k = k\pi /2$, which in turn reduces to the case of Alexander's spirals \eqref{alex}. Thus no nonsymmetric spirals of the form claimed in Theorem~\ref{thm_main} can be obtained for even $M\geq 4$.

\section{Asymptotic expansions near $a=\infty $}\label{sec_expans}

In this section we prove asymptotic expansions \eqref{expan1}--\eqref{expan2}, namely
\[\begin{split}
\Theta(a) &= \overline{\Theta}   +\Theta_{-1} a^{-1} + o(a^{-1}), \\
G(a,\Theta(a)) &=  \overline{G} (\overline{\Theta })  +  G_{-1} a^{-1} + o(a^{-1}),
\end{split}\]
as  $a\to \infty$, where $\Theta \colon [a_0 ,\infty )\to \R^{M-1}$ is the solution map obtained in Step 4, that is $F(a,\Theta (a) )=0$ and $\Theta(a) \to \overline{\Theta }$ as $a\to\infty$, and 
\[
\begin{split}
\Theta_{-1} &\coloneqq -\nabla \overline{F}(\overline{\Theta })^{-1} I (\overline{\Theta }),\\
G_{-1} &\coloneqq R (\overline{\Theta }) +  \nabla \overline{G}(\overline{\Theta } )\Theta_{-1},
\end{split}
\]
recall \eqref{def_of_-1}.

To this end we differentiate $F(a,\Theta(a))=0$ with respect to $a$ to obtain 
\begin{align*}
\partial_{a}F(a,\Theta(a)) + \nabla F(a,\Theta(a))\partial_{a}\Theta(a) = 0, \quad a>a_{0}.
\end{align*}
Thus the asymptotic expansion \eqref{expan1} of $\Theta (a)$ can be obtained by an application of l'H\^opital rule,   
\begin{equation*}
\lim_{a\to\infty} a(\Theta(a) -\overline{\Theta} ) = \lim_{a\to\infty}(-a^{2} \partial_{a}\Theta(a)) = -\nabla \overline{F}(\overline{\Theta} )^{-1}\lim_{a\to\infty}(-a^{2} \partial_{a}F(a, \Theta (a) )) = \Theta_{-1},
\end{equation*}
where we recalled \eqref{R+iI} that the last limit equals $I(\overline{\Theta })$. 

As for the expansion \eqref{expan2} we first note the chain rule
\begin{align*}
\frac{\d }{\d a } G(a,\Theta(a)) = \partial_{a}G(a,\Theta(a)) + \nabla G(a,\Theta(a))\partial_{a}\Theta(a), \quad a>a_{0},
\end{align*}
and we use l'H\^opital rule again to obtain
\begin{equation*}
\begin{split}
  \lim_{a\to\infty} a(G(a,\Theta(a)) - \overline{G}(\overline{\Theta }) ) &= \lim_{a\to\infty}\left( -a^{2} \frac{\d }{\d a }G(a,\Theta(a)) \right) \\
&= \lim_{a\to\infty}(-a^{2}\partial_{a}G(a,{\Theta (a)})) + \nabla \overline{G} (\overline{\Theta }) \lim_{a\to\infty}(-a^{2} \partial_{a}\Theta(a)) \\
& = R(\overline{\Theta } ) + \nabla \overline{G}(\overline{\Theta } )\Theta_{-1} = G_{-1},
\end{split}
\end{equation*}
as required, where have also recalled \eqref{R+iI} for the definition of $R(\overline{\Theta })$.

\section{Linear independence of $E_1$, $E_2$}\label{sec_choice_g_mu}

As mentioned in Step 5 in the introduction, here we show linear independence of $E_1(a)$ and $E_2(a)$ (recall \eqref{eq-g-n_repeat1}), provided the nondegeneracy condition \eqref{naF_is_inver}, that is invertibility of $\nabla \overline{F} (\overline{\Theta })$. The case $M=2$ can be verified directly (Section~\ref{sec_caseM2}), while the case of odd $M\geq 3$ can be settled by expansion as $a\to \infty$ (Section~\ref{sec_case_Mgeq3}).

\subsection{Case $M=2$}\label{sec_caseM2}

In this case we have
\[
\H_l (a, \theta (a) ) =\cosh (\pi A) - \ee^{A(\pi - \theta (a))}\quad \text{ and } \quad \K (a,\theta (a) )= \cosh(\pi A) - \ee^{A(\theta(a) -\pi )}
\]
(recall \eqref{def_of_K,Hl}), where we have set $\theta (a) \coloneqq \Theta (a) = \theta_1 (a) $, for simplicity. We recall the definitions of $E_1(a),E_2(a)$ from \eqref{eq-g-n_repeat1}, i.e.
\[\begin{split}
E_1 (a)& = \frac{a}{\sinh (\pi A) } \left( \cosh^2 (\pi A ) - 1 \right) = a \sinh (\pi A),\\
E_2 (a) & = \left( i - a^{-1} \right) \left( \cosh (\pi A) - \ee^{A(\theta (a) - \pi )} \right). 
\end{split}
\]
 Noting that
\[\begin{split}
\ee^{\pm A\pi }&= \ee^{\pm A\pi \pm 2i\pi}= \ee^{\mp \frac{2\pi }{a+i}} = 1\mp 2\pi a^{-1} + O(a^{-2}),\\
\ee^{A\theta (a) }&= \ee^{-\frac{ai\pi}{a+i} + O(a^{-1})}=  \ee^{-\pi i  + O(a^{-1})}=  -1 +O(a^{-1})
\end{split}
\]
as $a\to \infty$ (due to the fact that $\theta(a) = \pi/2 + O(a^{-1})$, recall \eqref{expan1}), we have 
\[\sinh (A\pi ) = -2\pi a^{-1} + O(a^{-2})\quad \text{ and }\quad \cosh( A\pi ) = 1+O(a^{-2}) \qquad \text{ as }a\to \infty.
\]
Thus, as $a\to \infty$,
\[
\begin{split}
E_1 (a) &= -2\pi + o(1),\\
E_2 (a) &= (i -a^{-1} ) \left( 1+O(a^{-2}) -(-1+O(a^{-1}))(1+2\pi a^{-1} + o (a^{-1})) \right)  = 2i + O(a^{-1}),
\end{split}
\]
which are linearly independent for sufficiently large $a>0$.

\subsection{The case of odd $M\geq 3$}\label{sec_case_Mgeq3}

Here we show that, for sufficiently large $a>0$, and odd $M\geq 3$,
\begin{gather*}
E_{1}(a) = \frac{a}{\sinh(\pi A)}\left(\cosh(\pi A)\mathcal{K}(a,\Theta(a))+ \sum_{l=1}^{M-1} \mathcal{H}_{l}(a,\Theta(a))e^{A(\theta_{l}(a)-\pi)} \right),\\
E_{2}(a) = \left(i - \frac{1}{a}\right)\mathcal{K}(a,\Theta(a))
\end{gather*}
(recall \eqref{eq-g-n_repeat1}) are linearly independent vectors in the complex plane $\C$ given \eqref{naF_is_inver} holds.
We first simplify the expression for $E_1 (a)$ and find an expansion for $E_2 (a)$.
\begin{proposition}\label{prop_aa}
We have 
\eqnb\label{asym-11aa}
\begin{split}
E_1 (a) &= a \cosh(\pi A) = a + 2\pi^2 a^{-1} + o(a^{-1}), \\
E_{2}(a) &= E_{2,0}+E_{2,-1}a^{-1} + o(a^{-1})
\end{split}
\eqne
as $a\to \infty$, where  $E_{2,0}\coloneqq 2\sin(2\theta_{1})(1+\ee^{-2i\theta_{1}})^{-1}(1+\ee^{2i\theta_{1}})^{-1}$ and 
\begin{align}\label{def-e}
E_{2,-1} \coloneqq -\overline{\mathcal{K}}(\overline{\Theta}) + i\lim_{a\to \infty}[-a^{2}\partial_{a}\mathcal{K}(a,\overline{\Theta})]
+i\lim_{a\to \infty}\nabla_{\Theta}\mathcal{K}(a,\overline{\Theta})\cdot \lim_{a\to \infty}[-a^{2}\partial_{a}\Theta(a)].
\end{align}
In particular, $E_1(a)$, $E_2 (a)$ are linearly independent for sufficiently large $a>0$ if $\im \, E_{2,-1} \ne 0$.
\end{proposition}
\begin{proof}
Let us observe that by \eqref{def_of_K,Hl} we have
\begin{equation}
\begin{aligned}\label{eq-11}
&\cosh(\pi A)\mathcal{K}(a,\Theta(a))+ \sum_{l=1}^{M-1} \mathcal{H}_{l}(a,\Theta(a))\ee^{A(\theta_{l}(a)-\pi)} \\
&= \cosh(\pi A)\sinh(\pi A) - \cosh(\pi A)\left(\sum_{k=1}^{M-1}(-1)^{k}\ee^{A(\theta_{k}(a)-\pi)}\right) + \cosh(\pi A)\left(\sum_{k=1}^{M-1}\ee^{A(\theta_{k}(a)-\pi)}\right) \\
&\qquad +\sum_{l=1}^{M-1}\sum_{k=0}^{l-1}(-1)^{l+k}\ee^{A\theta_{k}(a)} - \sum_{l=1}^{M-1}\sum_{k=l}^{M-1}(-1)^{l+k}\ee^{A(\theta_{k}(a)-2\pi)}.
\end{aligned}
\end{equation}
Since $M\ge 3$ is an odd number, we obtain
\begin{equation}
\begin{aligned}\label{eq-22}
&\sum_{l=1}^{M-1}\sum_{k=0}^{l-1}(-1)^{l+k}\ee^{A\theta_{k}(a)} = \sum_{l=1}^{M-1}\sum_{k=0}^{M-1} 1_{\{0\le k\le l-1\}}(k)(-1)^{l+k}\ee^{A\theta_{k}(a)} \\
& \qquad = \sum_{k=0}^{M-1} (-1)^{k}\ee^{A\theta_{k}(a)} \sum_{l=k+1}^{M-1}(-1)^{l} = -\sum_{k=0, \, 2\nmid k}^{M-1} \ee^{A\theta_{k}(a)}
\end{aligned}
\end{equation}
and similarly
\begin{equation}
\begin{aligned}\label{eq-33}
&\sum_{l=1}^{M-1}\sum_{k=l}^{M-1}(-1)^{l+k}\ee^{A(\theta_{k}(a)-2\pi)} = \sum_{l=1}^{M-1}\sum_{k=1}^{M-1}1_{\{l\le k\le M-1\}}(k)(-1)^{l+k}\ee^{A(\theta_{k}(a)-2\pi)} \\
&\qquad= \sum_{k=1}^{M-1}(-1)^{k}\ee^{A(\theta_{k}(a)-2\pi)}\sum_{l=1}^{k}(-1)^{l} 
= \sum_{k=1,\, 2\nmid k}^{M-1}\ee^{A(\theta_{k}(a)-2\pi)}.
\end{aligned}
\end{equation}
Combining \eqref{eq-11}, \eqref{eq-22} and \eqref{eq-33} gives 
\begin{align*}
& \cosh(\pi A)\mathcal{K}(a,\Theta(a))+ \sum_{l=1}^{M-1} \mathcal{H}_{l}(a,\Theta(a))\ee^{A(\theta_{l}(a)-\pi)} \\
&\qquad =\cosh(\pi A)\sinh(\pi A) + 2\cosh(\pi A)\left(\sum_{k=1,\, 2\nmid k}^{M-1}\ee^{A(\theta_{k}(a)-\pi)}\right) 
 -\sum_{k=0, \, 2\nmid k}^{M-1} \ee^{A\theta_{k}(a)} -\sum_{k=1,\, 2\nmid k}^{M-1}\ee^{A(\theta_{k}(a)-2\pi)} \\
&\qquad = \cosh(\pi A)\sinh(\pi A) + 2\cosh(\pi A)\ee^{-A\pi} \left(\sum_{k=1,\, 2\nmid k}^{M-1}\ee^{A\theta_{k}(a)}\right) 
 -\sum_{k=0, \, 2\nmid k}^{M-1}\ee^{A\theta_{k}(a)} -\ee^{-2A\pi} \sum_{k=1,\, 2\nmid k}^{M-1}\ee^{A\theta_{k}(a)} \\
&\qquad = \cosh(\pi A)\sinh(\pi A)+ \left(2\cosh(\pi A)\ee^{-A\pi} -1 -\ee^{-2A\pi} \right)\left(\sum_{k=1,\, 2\nmid k}^{M-1}\ee^{A\theta_{k}(a)}\right) \\
&\qquad = \cosh(\pi A)\sinh(\pi A),
\end{align*}
which implies that $E_{1}(a) = a\cosh(\pi A)$, and so, since $A\pi = -2i \pi -2\pi / (a+i)$,
 \[
\begin{split}
E_1(a) & = a \cosh (A\pi) = a \cosh (-2\pi/(a+i)) \\
& = a( 1+ 2\pi^{2}/(a+i)^2 + O((a+i)^{-4}) \\
& = a + 2\pi^{2}a^{-1} + o(a^{-1})
\end{split}
\] 
as $a\to \infty$, as required. 

As for $E_2 (a)$, since $\mathcal{K}(a,\Theta(a))\to \overline{\mathcal{K}}(\overline{\Theta})$ by the equation \eqref{rr3}, we have 
\begin{align*}
E_{2}(a)\to E_{2,0} = i \overline{\mathcal{K}}(\overline{\Theta}) =\frac{2\sin(2\theta_{1})}{(1+\ee^{-2i\theta_{1}})(1+\ee^{2i\theta_{1}})},\quad a\to \infty.
\end{align*}
We note that
\begin{align*}
\partial_{a}E_{2}(a) & =\frac{1}{a^{2}}\mathcal{K}(a,\Theta(a)) + \left(i-\frac{1}{a}\right)\partial_{a}\mathcal{K}(a,\Theta(a))
+\left(i-\frac{1}{a}\right)\nabla_{\theta}\mathcal{K}(a,\Theta(a))\cdot\partial_{a}\Theta(a)
\end{align*}
and consequently
\begin{align*}
E_{2,-1}& =\lim_{a\to \infty} a(E_{2}(a) - E_{2,0}) = \lim_{a\to \infty}[-a^{2}\partial_{a}E_{2}(a)] \\
&= -\overline{\mathcal{K}}(\overline{\Theta}) + i\lim_{a\to \infty}[-a^{2}\partial_{a}\mathcal{K}(a,\overline{\Theta})]
+i\nabla_{\theta}\mathcal{K}(a,\overline{\Theta})\cdot \lim_{a\to \infty}[-a^{2}\partial_{a}\Theta(a)] ,
\end{align*}
as desired. The last claim of the proposition follows from the fact that $E_{2,0}\in \R$. 
\end{proof}
Thanks to the above proposition we obtain the required linear independence of $E_1(a)$, $E_2 (a)$ by showing that $\mathrm{Im}\,E_{2,-1} \neq 0$.

\begin{proposition}\label{prop-e_2_1_n_0}
Given $M\geq 3$ is an odd integer such that \eqref{naF_is_inver} holds then $\mathrm{Im}\, E_{2,-1} \neq 0$. 
\end{proposition}

Before we prove Proposition~\ref{prop-e_2_1_n_0}, we recall that a complex number is algebraic provided it is a root of a non-zero polynomial with rational coefficients. We recall that the set $\mathbb{A}$ of algebraic numbers equipped with complex addition and multiplication is a field, which is closed under conjugation $\overline{z}$ for $z\in\C$ (see \cite[Theorem 50]{MR0638719} or \cite[p. 174]{MR0356988}). Therefore, if $z\in\mathbb{A}$ then both its real part $\mathrm{Re}\, z = (z+\overline{z})/2$ and imaginary part $\mathrm{Im}\, z = (z-\overline{z})/2i$  are also elements of the field. Observe also that $\mathbb{A}$ contains the field $\Q$ of the rational numbers.

\begin{proof}[Proof of Proposition~\ref{prop-e_2_1_n_0}.]
Suppose that  $\mathrm{Im}\,E_{2,-1} \ne 0$. We first show that 
\begin{align}\label{eq-alg-1}
i\lim_{a\to \infty}[-a^{2}\partial_{a}\mathcal{K}(a,\overline{\Theta})]
+i\lim_{a\to \infty}\nabla_{\theta}\mathcal{K}(a,\overline{\Theta})\cdot\Theta_{-1} \in\pi\mathbb{A}.
\end{align}
Indeed, let us observe that $\ee^{i\overline{\theta}_k}=\ee^{kn\pi i/M}\in \mathbb{A}$ for $k\in \{ 1,\ldots , M-1 \} $ and $n\in\{1,2\}$ as a root of the polynomial $z^{2M}-1$. Therefore $\cos(\overline{\theta}_k) = \mathrm{Re}\,\ee^{i\overline{\theta}_k}$ and $\sin(\overline{\theta}_k) = \mathrm{Im}\,\ee^{i\overline{\theta}_k}$  are also elements of $\mathbb{A}$. In particular, by  \eqref{grad_Finfty}, the entries of the gradient matrix $\nabla\overline{F}(\overline{\Theta})$ are also algebraic numbers. Consequently, since this matrix is invertible by assumption \eqref{naF_is_inver}, the  definition of the inverse matrix shows that the entries of $\nabla\overline{F}(\overline{\Theta})^{-1}$ are algebraic as well. From  \eqref{limit_at_angles} it follows that $I(\overline{\Theta})/\pi \in\mathbb{A}$, which in turn gives 
\begin{align}\label{eq-alg}
\Theta_{-1}/\pi = -\nabla \overline{F}(\overline{\Theta })^{-1} I (\overline{\Theta }) /\pi \in \mathbb{A}.
\end{align}
Moreover \eqref{def_of_K,Hl} shows that 
\begin{align*}
\lim_{a\to \infty}\nabla_{\theta_{k}}\mathcal{K}(a,\overline{\Theta}) = 2i (-1)^{k}\ee^{-2i\overline{\theta}_{k}},\quad 1\le k\le M-1,
\end{align*}
which, combined with \eqref{eq-alg}, yields 
\begin{align}\label{eq-alg-2}
i\lim_{a\to \infty}\nabla_{\theta}\mathcal{K}(a,\overline{\Theta})\cdot \Theta_{-1}/\pi \in\mathbb{A}.
\end{align}
Furthermore, by \eqref{asy-k} we have $i\lim_{a\to\infty}\left[-a^{2}\partial_{a} \mathcal{K} (a,\overline{\Theta } )\right]/\pi \in\mathbb{A}$, which together with \eqref{eq-alg-2} gives \eqref{eq-alg-1}, as desired. Hence, in view of the assumption $\mathrm{Im}\,E_{2,-1} = 0$ and the equation \eqref{def-e}, we have
\begin{align*}
\mathrm{Im}\,\overline{\mathcal{K}}(\overline{\Theta})  = \mathrm{Im}\,\left(i\lim_{a\to \infty}[-a^{2}\partial_{a}\mathcal{K}(a,\overline{\Theta})]
+i\lim_{a\to \infty}\nabla_{\theta}\mathcal{K}(a,\overline{\Theta})\cdot \lim_{a\to \infty}[-a^{2}\partial_{a}\Theta(a)]\right) \in\pi\mathbb{A}.
\end{align*}
On the other hand,  \eqref{rr3} implies that
\begin{align*}
\mathrm{Im}\, \overline{\mathcal{K}}(\overline{\Theta}) = -\frac{2\sin(2\theta_{1})}{(1+\ee^{-2i\theta_{1}})(1+\ee^{2i\theta_{1}})}\in\mathbb{A}\setminus\{0\},
\end{align*}
which gives $\pi\in\mathbb{A}$. This is a contradiction because $\pi$ is a transcendental number by the Lindemann-Weierstrass theorem (see \cite{MR1510165}, \cite{Weierstrass}). 
\end{proof}

\section{Appendix}

In this section we verify some statements referred to in the introduction.

\subsection{Appendix - Invertibility of $\mathcal{A}$}\label{sec_app1}

Here we show that, if $r\in \C \setminus \{ 0,\infty \}$, $M\in \N$, $r_1, \ldots , r_M \in \C \setminus \{ 0, \infty \}$, and 
\[
\mathcal{A} \coloneqq \begin{pmatrix}  
\frac{r+r^{-1}}2  & r \frac{r_1 }{r_2} & r \frac{r_1}{r_3} & \ldots & r \frac{r_1}{r_M} \\ r^{-1} \frac{r_2}{r_1} & \frac{r + r^{-1}}2 & r \frac{r_2}{r_3} &\ldots & r \frac{r_2}{r_M}  \\
r^{-1} \frac{r_3}{r_1} & r^{-1} \frac{r_3}{r_2} & \frac{r+r^{-1} }2 & \ldots & r \frac{r_3}{r_M} \\
\vdots & \vdots & \vdots & \ddots & \vdots \\
r^{-1} \frac{r_M}{r_1} & r^{-1} \frac{r_M}{r_2} &r^{-1} \frac{r_M}{r_3} &\ldots & \frac{r+r^{-1}}2
\end{pmatrix}
\]
then
\eqnb\label{det_of_A}
\det \mathcal{A} = \left( \frac{r-r^{-1}}2 \right)^{2p} \left( \frac{r+r^{-1}}2 \right)^q,
\eqne
where $M=2p+q$, $q\in \{ 0,1\}$. In particular $\mathcal{A}$ is invertible if and only if $r\not \in \{ 1,-1,i,-i \}$. Thus taking $r\coloneqq \ee^{\pi A}$, where $A=-2ai/(a+i)$ (recall \eqref{def_of_A}), we see that $\frac{2A}{i} = \frac{-4a}{a+i}$ is not an integer for any $a\in (0,\infty )$, and therefore \eqref{matrix-a} is invertible for every $a\in (0,\infty )$, as mentioned in the introduction (below \eqref{def_r_rk}).\\

In order to prove \eqref{det_of_A} we we will show that $\mathcal{A}$ admits the following $LU$ decomposition, $\mathcal{A} =LU$, where
\[
L\coloneqq   \begin{pmatrix}  
1  & 0  & 0 & 0& \ldots & 0 \\ \frac{r_2}{r_1} \frac{2r^{-1}}{r+r^{-1}} & 1 & 0 &0&\ldots & 0  \\
 \frac{r_3}{r_1} \frac{2r^{-1}}{r+r^{-1}}&  \frac{r_3}{r_2} \frac{2(r^{-1}+1)(r^{-1}-1)}{(r-r^{-1})^2}   &1  &0& \ldots & 0 \\
  \frac{r_4}{r_1} \frac{2r^{-1}}{r+r^{-1}} & \frac{r_4}{r_2} \frac{2(r^{-1}+1)(r^{-1}-1)}{(r+r^{-1})^2} & \frac{r_4}{r_3} \frac{2r^{-1} }{r+r^{-1}}& 1 & \ldots & 0\\
\vdots & \vdots & \vdots &\vdots & \ddots & \vdots \\
 \frac{r_M}{r_1} \frac{2r^{-1}}{r+r^{-1}} & \frac{r_M}{r_2} \frac{2(r^{-1}+1)(r^{-1}-1)}{(r+r^{-1})^2} & \frac{r_M}{r_3} \frac{2r^{-1} }{r+r^{-1}}& \frac{r_M}{r_4} \frac{2(r^{-1}+1)(r^{-1}-1)}{(r+r^{-1})^2} & \ldots & 1
\end{pmatrix}
\]
and
\[U\coloneqq 
 \begin{pmatrix}  
\frac{r+r^{-1}}2  &  \frac{r_1 }{r_2}r &  \frac{r_1}{r_3}r & \frac{r_1}{r_4} r& \ldots &  \frac{r_1}{r_M} r\\  0&  \frac{(r-r^{-1})^2}{2(r+r^{-1})} &  \frac{r_2}{r_3} \frac{(r+1)(r-1)}{r+r^{-1}} &\frac{r_2}{r_4} \frac{(r+1)(r-1)}{r+r^{-1}} &\ldots &  \frac{r_2}{r_M} \frac{(r+1)(r-1)}{r+r^{-1}}  \\
0 & 0 & \frac{r+r^{-1} }2 &  \frac{r_3}{r_4} r& \ldots &  \frac{r_3}{r_M} r\\
0 & 0 & 0  &  \frac{(r-r^{-1})^2}{2(r+r^{-1})}&\ldots & \frac{r_4}{r_M} \frac{(r+1)(r-1)}{r+r^{-1}}  \\
\vdots & \vdots & \vdots & \vdots & \ddots & \vdots \\
0 & 0 &0 &0 &\ldots & \frac{r+r^{-1}}2 
\end{pmatrix}
\]
if $M$ is odd. (Otherwise, if $M$ is even, the last term in $U$ is replaced by $\frac{(r-r^{-1})^2}{2(r+r^{-1})}$.)

In other words
\[
L_{kj} = \begin{cases} 
0\hspace{2cm} &j>k,\\
1& j=k,\\
\frac{r_k}{r_j}\frac{2r^{-1}}{r+r^{-1}} & j<k, j\text{ odd},\\
\frac{r_k}{r_j}\frac{2(r^{-1}+1)(r^{-1}-1)}{(r-r^{-1})^2} & j<k, j\text{ even},
\end{cases}\quad U_{jm}=\begin{cases} 
0\hspace{2cm} &j>m,\\
\frac{r+r^{-1}}2& j=m \text{ is odd},\\
\frac{(r-r^{-1})^2}{2(r+r^{-1})}& j=m \text{ is even},\\
\frac{r_j}{r_m}r  & j<m, j\text{ odd},\\
\frac{r_j}{r_m}\frac{(r+1)(r-1)}{r+r^{-1}} & j<m, j\text{ even},
\end{cases}
\]
We now verify that $(LU)_{km}=\mathcal{A}_{km}$ for every pair of $k,m=1,\ldots , M$.\\

We first note the algebraic identity
\eqnb\label{alg_id_for_matrices}
\frac{(r^{-1}+1)(r^{-1}-1)(r+1)(r-1)}{(r-r^{-1})^2}  = -1.
\eqne

\noindent\texttt{Case 1.} If $k=m$ is odd then
\[
L_{kj} U_{jm} = \sum_{\substack{j<k \\
                  j\text{ odd} }} \frac{r_k}{r_m} \frac{2 r^{-1}}{r+r^{-1}} r + \sum_{\substack{j<k \\
                  j\text{ even} }} \frac{r_k}{r_m} \frac{2(r^{-1}+1)(r^{-1}-1)(r+1)(r-1)}{(r-r^{-1})^2(r+r^{-1})} + \frac{r+r^{-1}}2 .
\]
Noting that the first two sums cancel due to the identity \eqref{alg_id_for_matrices} and the fact that $k$ is odd (so that there is the same number of odd and even $j<k$), we obtain $(r+r^{-1})/2 = \mathcal{A}_{km}$, as required.\\

\noindent\texttt{Case 2.} If $k=m$ is even then
\[\begin{split}
L_{kj} U_{jm} &= \sum_{\substack{j<k \\
                  j\text{ odd} }} \frac{r_k}{r_m} \frac{2 r^{-1}}{r+r^{-1}} r + \sum_{\substack{j<k \\
                  j\text{ even} }} \frac{r_k}{r_m} \frac{2(r^{-1}+1)(r^{-1}-1)(r+1)(r-1)}{(r-r^{-1})^2(r+r^{-1})} + \frac{(r-r^{-1})^2}{2(r+r^{-1})} \\
                  &= \frac{2}{r+r^{-1}} + \frac{(r-r^{-1})^2}{2(r+r^{-1})} = \frac{r+r^{-1}}2=\mathcal{A}_{km},
\end{split}\]
where, in comparison to Case 1, we now have one more odd $j$ (than we have even $j$'s, as $k$ is even). \\

\noindent\texttt{Case 3.} If $k<m$ and $k$ is odd then
\[
L_{kj} U_{jm} = \sum_{\substack{j<k \\
                  j\text{ odd} }} \frac{r_k}{r_m} \frac{2 r^{-1}}{r+r^{-1}} r + \sum_{\substack{j<k \\
                  j\text{ even} }} \frac{r_k}{r_m} \frac{2(r^{-1}+1)(r^{-1}-1)(r+1)(r-1)}{(r-r^{-1})^2(r+r^{-1})} + \frac{r_k}{r_m}r =  \frac{r_k}{r_m}r =  \mathcal{A}_{km},
\]
as required. \\

\noindent\texttt{Case 4.} If $k<m$ and $k$ is even then
\[\begin{split}
L_{kj} U_{jm} &= \sum_{\substack{j<k \\
                  j\text{ odd} }} \frac{r_k}{r_m} \frac{2 r^{-1}}{r+r^{-1}} r + \sum_{\substack{j<k \\
                  j\text{ even} }} \frac{r_k}{r_m} \frac{2(r^{-1}+1)(r^{-1}-1)(r+1)(r-1)}{(r-r^{-1})^2(r+r^{-1})} + \frac{r_k}{r_m}\frac{(r-1)(r+1)}{r+r^{-1}} \\
                  &= \frac{r_k}{r_m} \left( \frac{2}{r+r^{-1}} + \frac{(r-1)(r+1)}{r+r^{-1}}\right)= \frac{r_k}{r_m} r = \mathcal{A}_{km},
                  \end{split}
\]
as required. \\

\noindent\texttt{Case 5.} If $k>m$ and $m$ is odd then
\[\begin{split}
L_{kj} U_{jm} &= \sum_{\substack{j<m \\
                  j\text{ odd} }} \frac{r_k}{r_m} \frac{2 r^{-1}}{r+r^{-1}} r + \sum_{\substack{j<m \\
                  j\text{ even} }} \frac{r_k}{r_m} \frac{2(r^{-1}+1)(r^{-1}-1)(r+1)(r-1)}{(r-r^{-1})^2(r+r^{-1})} \\
                  &\hspace{1cm}+ \frac{r_k}{r_m} \frac{2r^{-1}}{r+r^{-1}} \frac{r+r^{-1}}2 =  \frac{r_k}{r_m}r^{-1} =  \mathcal{A}_{km},
                  \end{split}
\]
as required. \\

\noindent\texttt{Case 6.} If $k>m$ and $m$ is even then
\[\begin{split}
L_{kj} U_{jm} &= \sum_{\substack{j<m \\
                  j\text{ odd} }} \frac{r_k}{r_m} \frac{2 r^{-1}}{r+r^{-1}} r + \sum_{\substack{j<m \\
                  j\text{ even} }} \frac{r_k}{r_m} \frac{2(r^{-1}+1)(r^{-1}-1)(r+1)(r-1)}{(r-r^{-1})^2(r+r^{-1})} \\
                  &\hspace{1cm}+ \frac{r_k}{r_m} \frac{2(r^{-1}+1)(r^{-1}-1)}{(r-r^{-1})^2} \frac{(r-r^{-1})^2}{2(r+r^{-1})}\\
                  &= \frac{r_k}{r_m} \left( \frac{2}{r+r^{-1}}+ \frac{(r^{-1}+1)(r^{-1}-1)}{r+r^{-1}}\right)=\frac{r_k}{r_m} r^{-1} = \mathcal{A}_{km},
                  \end{split}
\]
as required. \hfill $\square$

\subsection{Appendix - Nontriviality of nonsymmetric spirals}\label{sec_app2}

Here we verify Remark~\ref{rem_nontrivial}. Namely, we show that the nonsymmetric spirals obtained in Theorem~\ref{thm_main} are nontrivial in the sense that there are no $a>0$, $\mu \in \R$ such that $\Theta \coloneqq \overline{\Theta} = \frac{\pi}M (1,2,\ldots, M-1)$ (recall \eqref{angles_possible}), and $g_k=g\in \R\setminus \{ 0 \} $ is a solution to \eqref{eq-disc}. 

Indeed
\begin{align*}
& \sum_{k=0}^{M-1}\ee^{A(\theta_{k}-\theta_{m})}
\left\{\begin{aligned}
&\ee^{-\pi A} && \text{if} \ \theta_{k}>\theta_{m},\\
&\cosh(\pi A) && \text{if} \ \theta_{k}=\theta_{m},\\
&\ee^{\pi A} && \text{if} \ \theta_{k}<\theta_{m},
\end{aligned}\right. \\
&\qquad =\sum_{k=0}^{m-1} \ee^{\pi A (k-m)/M}\ee^{\pi A} + \cosh(\pi A) + \sum_{k=m+1}^{M-1} \ee^{\pi A(k-m)/M}\ee^{-\pi A} \\
& \qquad =\ee^{-\pi A m/M}\frac{1-\ee^{\pi A m/M}}{1-\ee^{\pi A /M}}\ee^{\pi A} + \cosh(\pi A ) + \left(\frac{1-\ee^{\pi A (M-m)/M}}{1-\ee^{\pi A/M}}-1\right)\ee^{-\pi A} \\
&\qquad = \sinh (\pi A) \frac{1+\ee^{\pi A /M}}{\ee^{2\pi A /M}-1}+\ee^{-\pi A m/M}\frac{\ee^{\pi A}-1}{1-\ee^{\pi A/M}}.
\end{align*}
Thus \eqref{eq-disc} fails in the case $\Theta = \overline{\Theta}$ and $g_{k}=g$ as the right-hand side above takes different values for $m\in \{ 1, \ldots ,  M-1\}$.

\section*{Acknowledgments}
T.C. was partially supported by the National Science Centre grant SONATA BIS 7 number UMO-2017/26/E/ST1/00989. W.S.O. was supported in part by the Simons Foundation. 

\end{document}